\documentclass[11pt,a4paper,reqno]{amsart}
\usepackage{amsmath, amsthm, amsfonts, amssymb, mathrsfs, graphicx,xcolor, mathtools, caption, subcaption, enumerate, eucal, enumitem, marginnote}
\newtheorem{theorem}{Theorem}[section]
\newtheorem{lemma}[theorem]{Lemma}
\newtheorem{proposition}[theorem]{Proposition}
\newtheorem{corollary}[theorem]{Corollary}
\theoremstyle{remark}
\newtheorem{remark}[theorem]{Remark}
\theoremstyle{definition}
\newtheorem*{definition}{Definition}
\newtheorem{example}[theorem]{Example}
\newtheorem{ass}{Assumption}

\numberwithin{equation}{section}

\newcommand{\clr}{\textcolor{red}}

\newcommand{\NN}{{\mathbb N}}
\newcommand{\RR}{{\mathbb R}}
\newcommand{\CC}{{\mathbb C}}

\DeclareMathOperator\Ran{Ran}

\newcommand{\Dom}{{\operatorname{Dom}}}
\newcommand{\Ker}{{\operatorname{Ker}}}

\newcommand{\nl}{{\operatorname{nul}}}
\newcommand{\df}{{\operatorname{def}}}

\DeclareMathOperator\Real{Re}
\DeclareMathOperator\Imag{Im}
\renewcommand\Re{\Real}
\renewcommand\Im{\Imag}

\newcommand{\dist}{\operatorname{dist}}

\newcommand\la{\lambda}

\newcommand\wt{\widetilde}
\newcommand\wh{\widehat}

\newcommand\cA{{\mathcal A}}
\newcommand\cB{{\mathcal B}}

\newcommand\cH{{\mathcal H}}

\newcommand\sS{{\mathscr S}}
\newcommand\sC{{\mathscr C}}

\newcommand\dx{{\displaystyle{\frac{\d}{\d x}}}}

\newcommand\D{\partial}
\renewcommand\d{{\rm d}}

\newcommand{\ii}{{\rm i}}

\newcommand\ds{\displaystyle}

\newcommand{\ov}[1]{\overline{#1}}
\newcommand\mydot{\,\cdot\,}
\newcommand{\defeq}{\mathrel{\mathop:}=}

\newcommand\spt{\sigma_{\rm p}}
\newcommand\sess{\sigma_{\rm ess}}
\newcommand\sessreg{\sigma_{\rm ess}^{\rm \,r}}
\newcommand\sesssing{\sigma_{\rm ess}^{\rm \,s}}
\newcommand\sapp{\sigma_{\rm app}}

\newcommand\cl{\rm{cl}}

\begin{document}


\title[Essential spectrum of matrix differential operators]{Essential spectrum of non-self-adjoint \\singular matrix differential operators}

\author{Orif \,O.\ Ibrogimov}
\address[O.\,O.\ Ibrogimov]{%
	Mathematisches Institut, 
	Universit\"at Bern, 
	Alpeneggstr. 22, 3012 
	Bern, Switzerland}

\curraddr{Department of Mathematics, 
	University College London,
	Gower Street, London,
	WC1E 6BT, UK}

\email{orif.ibrogimov@math.unibe.ch, o.ibrogimov@ucl.ac.uk}

\date{\today}

\begin{abstract}
The purpose of this paper is to study the essential spectrum of non-self-adjoint singular matrix differential operators in the Hilbert space $L^2(\RR)\oplus L^2(\RR)$ induced by matrix differential expressions of the form
\begin{align}\label{abstract:mdo}
    \left(\begin{array}{cc} 
    \tau_{11}(\mydot,D) & \tau_{12}(\mydot,D)\\[3.5ex]
    \tau_{21}(\mydot,D) & \tau_{22}(\mydot,D)
    \end{array}\right),
\end{align} 
where $\tau_{11}$, $\tau_{12}$, $\tau_{21}$, $\tau_{22}$ are respectively $m$-th, $n$-th, $k$-th and 0 order ~ordinary differential expressions with $m=n+k$ being even. Under suitable assumptions on their coefficients, we establish an analytic description of the essential spectrum. It turns out that the points of the essential spectrum either have a local origin, which can be traced to points where the ellipticity in the sense of Douglis and Nirenberg breaks down, or they are caused by singularity at infinity. 
\end{abstract}

\maketitle

\section{Introduction} 

The spectral analysis of singular matrix differential operators generated by matrix differential expressions as in \eqref{abstract:mdo} is important in many branches of theoretical physics including magnetohydrodynamics and  astrophysics, see e.g.\ \cite{Lifschitz-89b}, \cite{BEY95}. For example, in models of linear stability theory of plasmas confined to a toroidal region in $\RR^3$, by eliminating one variable by means of the $S^1$-symmetry, one arrives at second order systems of partial differential equations in the radial and angular variables on the cross section of the torus. Using a Fourier series decomposition with respect to the angular variable, the operator matrix then becomes a direct sum of operators of the form \eqref{abstract:mdo}, 
see e.g.\ \cite{Decloux-Geymonat-80}, \cite{FMM04}. In view of linear stability analysis and numerical approximations, it is of crucial importance to have information on the location of the whole essential spectrum of operator matrices as in \eqref{abstract:mdo}.  

Two of the reasons why such matrix differential operators received continuous attention of specialists in spectral theory during the last thirty years may be explained as follows. Firstly, in contrast to the case of scalar differential operators, matrix differential operators need not to have empty essential spectrum even if the underlying domain is compact and the corresponding boundary conditions are regular. This is due to the matrix structure which allows for essential spectrum to arise because of the violation of ellipticity in an appropriate sense. Secondly and most interestingly, in the case when the underlying domain is not compact, the essential spectrum of the matrix differential operator cannot be approximated by the essential spectra of operators determined by the same operator matrix on an increasing sequence of compact sub-domains exhausting to the original domain. In fact, it turns out that the essential spectrum can have a branch caused because of the singularity at infinity or, more generally, at 
the boundary of a non-compact interval.

Essential spectrum of matrix differential operators generated by  \eqref{abstract:mdo} is well known if the underlying domain is compact, see \cite{ALMS94}. The situation is much more complicated if the underlying domain is non-compact. In this case the spectral properties are far from being fully-understood up to date in the non-self-adjoint setting, especially when the matrix differential operator is not a perturbation of a self-adjoint operator.

The appearance of the branch of essential spectrum due to the singularity at the boundary was first predicted by Descloux and Geymonat \cite{Decloux-Geymonat-80} in connection with a physical model describing the oscillations of plasma in an equilibrium configuration in a cylindrical domain and proven much later by Faierman, Mennicken and M\"oller\ ~\cite{FMM04}. Similar phenomena in connection with problems of theoretical physics have been studied by many authors including Kako~\cite{Kako84, Kako85, Kako87}, Descloux and Kako \cite{D-Kako91},  Raikov \cite{Raikov-90, Raikov-91}, Beyer \cite{BEY95}, Atkinson, H. Langer, Mennicken and Shkalikov \cite{ALMS94}, H. Langer and M\"oller \cite{LM96}, Faierman, Mennicken and M\"oller \cite{FMM99, FMM00}, Hardt, Mennicken and Naboko \cite{HMN99}, Konstantinov~\cite{Konstantinov-2002}, Mennicken, Naboko and Tretter \cite{MNT02}, Kurasov and Naboko \cite{KN02}, M\"oller \cite{Moeller04}, Marletta and Tretter \cite{Mt07}, Kurasov, Lelyavin and Naboko \cite{KLN08}, Qi and Chen \cite{QCh11, QCh11-b}. 

Most of these studies were concerned with the investigation of particular and ``almost-symmetric" operators and it was shown that the essential spectrum due to the singularity at the boundary appears because of a very special interplay between the matrix entries. The first analytic description the essential spectrum in the general setting was established in \cite{ILLT13} in the symmetric case with $m=2$, $n=k=1$ and $[0,\infty)$ instead of $\RR$. The results of \cite{ILLT13} were later extended to much wider classes of symmetric matrix differential operators under considerably weaker assumptions in \cite{IST15} where the second diagonal entry allowed to be a matrix multiplication operator.

The current manuscript seems to be the first attempt to investigate the essential spectrum, in particular, both above mentioned spectral phenomena, for \emph{non-self-adjoint} matrices of ordinary differential operators of mixed-orders on the real line. The aim is to establish an analytic description of the entire essential spectrum in terms of the coefficients of \eqref{abstract:mdo}. Our method to describe the part of the essential spectrum caused by the singularity at infinity analytically is different from the so-called ``cleaning of the resolvent" approach suggested in [20], which was based on a result on the 
essential spectrum of separable sum of pseudo-differential operators, see [20, Theorem A.1]. Nevertheless, the remarks in [20] concerning the non-self-adjoint case have inspired the present paper.


The paper is organized as follows. Section 2 provides the necessary operator theoretic framework for the matrix differential operator generated by \eqref{abstract:mdo} in the Hilbert space $L^2(\RR)\oplus L^2(\RR)$ as well as for its first Schur complement. Section ~3 is dedicated to the description of the essential spectrum due to the singula\-rity at infinity. It is characterized in terms of the essential spectrum of the first Schur complement using the characterization of Fredholm operators in terms of approximate/generalized inverses and pseudo-differential operator techniques. In Section 4 the essential spectrum due to the violation of ellipticity in the sense of Douglis and Nirenberg is described. Section 5 contains the main result of the paper (see Theorem~\ref{thm:main.result}), where by suitable gluing and smoothing the results of the two previous sections are blended together. The obtained analytic description of the whole essential spectrum is given in terms of the original coefficients of the matrix differential operator in \eqref{abstract:mdo} and illustrated by an example.  

The following notation is used throughout the paper. We write $\langle\mydot,\mydot\rangle$ for the inner product in $L^2(\RR)$. For a Banach space $X$, by $\sC(X)$, $\mathscr{B}(X)$ and $\mathscr{K}(X)$, we denote respectively the set of closed, bounded and compact linear operators acting from $X$ to itself. For $T\in\sC(X)$, we denote by $\Dom(T)$, $\Ker(T)$ and $\Ran(T)$ the domain, kernel and the range of $T$, respectively. For a densely defined operator $T\in\mathscr{C}(X)$, $\rho(T)$ and $\Pi(T)$ denote respectively its resolvent set and the regularity field; for the essential spectrum, we use the definition 
\begin{equation*}
    \sess (T) := \{ \la \in \CC : T-\la \text{ is not Fredholm}\},
\end{equation*}
which is the set $\sigma_{\rm{e}3}(T)$ in \cite[Section~IX.1]{EE87}. Recall that $T\in\sC(X)$ is called Fredholm if $\Ran(T)$ is closed and both $\nl(T):=\dim\Ker(T)$ and $\df(T):=\dim X/\Ran(T)$ are finite. An identity operator is denoted by $I$, and scalar multiples $\omega I$ for $\omega\in\CC$ are written as $\omega$.

Furthermore, $\sS(\RR)$ stands for the Schwartz space of rapidly decaying functions on $\RR$. For $s\in\RR$ and a subinterval $J\subset\RR$, we denote by $H^s(J)$ the $L^2$-Sobolev space of order $s$ and by $H^s_0(J)$ the closed linear subspace of $H^s(J)$ obtained by taking the closure of $C_0^\infty(J)$ in $H^s(J)$. By $B^\infty(\RR,\CC)$ we denote the space of infinitely smooth functions $f:\RR\to\CC$ with bounded derivatives of arbitrary order. For a subinterval $J\subset\RR$ and a function $f:J\to\CC$, we denote by $f(J)$ the image of $J$ under $f$ and by ${\cl}\{f(J)\}$ the closure of the set $f(J)$ in $\CC$. 

\section{Higher order matrix differential operators and associated \\first Schur complement}\label{sec:A.def}
Let $n$, $k$ be non-negative integers such that $m:=n+k\in 2\NN$. We introduce the differential expressions
\begin{equation}\label{diff.expressions}
\begin{aligned}
    \tau_{11}(x,D) &:= \sum_{\alpha=0}^{m} a_{\alpha} D^\alpha, \qquad &\tau_{12}(x,D) &:= \sum_{\beta=0}^{n} b_{\beta}D^\beta\\
    \tau_{21}(x,D) &:= \sum_{\gamma=0}^{k} c_{\gamma}D^\gamma, \qquad &\tau_{22}(x,D) &:= d,
\end{aligned}
\end{equation}
where $D:=-\ii \, {\rm{d}}/{\rm{d}}x$ is the momentum operator, and we assume that the coefficient functions satisfy the following hypotheses. 
\begin{ass}
\label{ass:A} 
$a_{\alpha}, b_{\beta}, c_{\gamma}, d\!\in\! C^{\infty}(\RR,\CC)$ for all $\alpha\in\{0,1,\ldots,m\}$, $\beta\in\{0,1,\ldots,n\}$, $\gamma\in\{0,1,\ldots,k\}$ and $a_m(x)\neq0$, $x\in\RR$.
\end{ass}

Let $A_0$, $B_0$, $C_0$ and $D_0$ be the operators in the Hilbert space $L^2(\RR)$ induced respectively by the differential expressions $\tau_{11}(\mydot,D)$, $\tau_{12}(\mydot,D)$, $\tau_{21}(\mydot,D)$ and $\tau_{22}(\mydot,D)$ with domains being $C_0^\infty(\RR)$.

In the Hilbert space $\cH := L^2(\RR) \oplus L^2(\RR)$, we introduce the matrix differential operator
\begin{equation}\label{A0}
\begin{aligned} 
    \cA_0 & := \begin{pmatrix} A_0 & B_0 \\ C_0 & D_0 \end{pmatrix} :=
    \left(\begin{array}{cc} 
    \ds\sum_{\alpha=0}^{m} a_{\alpha}D^\alpha \: & \ds\sum_{\beta=0}^{n} b_{\beta}D^\beta\\[4ex]
    \ds\sum_{\gamma=0}^{k} c_{\gamma}D^\gamma \: & d 
    \end{array}\right)
\end{aligned}
\end{equation}
on the domain 
    \begin{equation}\label{domA0}
    \Dom(\cA_0) := C_0^\infty(\RR) \oplus C_0^\infty(\RR).
    \end{equation}
An easy integration by parts argument shows that the domain of the adjoint of $\cA_0$ contains  $C_0^\infty(\RR) \oplus C_0^\infty(\RR)$, and hence $\Dom(\cA_0^*)$ is dense in $\cH$. Therefore, $\cA_0$ is closable; we denote the closure of  $\cA_0$ by $\cA$.

Schur complements are useful tools in studying spectral properties of operator matrices, see \cite[Section 2.2]{Tre08}. Formally, the (first) Schur complement of the operator matrix $\cA_0$ in \eqref{A0} is given by $A_0-\la-B_0(D_0-\la)^{-1}C_0$ in $L^2(\RR)$ for $\la\in\rho(\ov{D_0})=\CC\setminus{\cl}\{d(\RR)\}$. Hence, for $\la\in\CC\setminus{\cl}\{d(\RR)\}$ it is induced by the $m$-th order scalar differential expression   
\begin{align*}
    \tau_S(\la) & := \tau_{11}-\la-\tau_{12}(\tau_{22}-\la)^{-1}\tau_{21}\\
                &= \sum_{\alpha=0}^{m} a_{\alpha}D^\alpha - \la  - \Bigl(\sum_{\beta=0}^{n} b_{\beta}D^\beta \Bigr) 
                \frac{1}{d-\la} \Bigl(\ds\sum_{\gamma=0}^{k} c_{\gamma}D^\gamma \Bigr).
\end{align*}    
The differential expression $\tau_S(\la)$ can be rewritten in the standard form as 
\begin{equation}  
\ds\tau_S(\la)=\sum_{j=0}^m p_j(\mydot,\la)D^j
\end{equation}
where, for $\la\in\CC\setminus{\cl}\{d(\RR)\}$, the coefficient functions are given by
\begin{equation}\label{Schur.coeff}
    p_j(\mydot,\la) := \begin{cases} 
                     \ds a_0-\la-\frac{b_0c_0}{d-\la}-\sum_{\beta=0}^n  b_{\beta}\frac{\D^{\beta}}{\D x^{\beta}}\Bigl(\frac{c_0}{d-\la}\Bigr) & \mbox{for} \quad j=0, \\[3ex]
                     \ds a_j-\sum_{\beta=0}^{\min(n,j)}\frac{b_{\beta}c_{j-\beta}}{d-\la} - \sum_{\beta=0}^n  b_{\beta}\frac{\D^{\beta}}{\D x^{\beta}}\Bigl(\frac{c_j}{d-\la}\Bigr) & \mbox{for} \quad 1\leq j \leq k, \\[3ex]
                     \ds a_j-\sum_{\beta=j-k}^{\min(n,j)}\frac{b_{\beta}c_{j-\beta}}{d-\la} & \mbox{for} \quad k<j\leq m.
                    \end{cases}
\end{equation}
For $\la\in\CC\setminus{\cl}\{d(\RR)\}$, the differential expression $\tau_S(\la)$ induces an operator $S_0(\la)$ in $L^2(\RR)$ on the Schwartz space,
\[
S_0(\la)u := \tau_S(\la)u, \quad u\in\Dom(S_0(\la)) := \mathscr{S}(\RR). 
\]
Moreover, $S_0(\la)$ is closable since $\mathscr{S}(\RR)$ is contained in the domain of  adjoint of $S_0(\la)$ and the former is dense in $L^2(\RR)$; we denote the closure of $S_0(\la)$ by $S(\la)$.

Observe that, by \eqref{Schur.coeff}, for $\la\in\CC\setminus{\cl}\{d(\RR)\}$, the leading coefficient of $S_0(\la)$ is given by
\begin{align}\label{leading-coefficient-Schur}
p_m(\mydot,\la) = a_m-\frac{b_nc_k}{d-\la}=a_m\frac{\Delta-\la}{d-\la},
\end{align}
where $\Delta:\RR\to\CC$ is defined as 
\begin{align}\label{defDelta}
\Delta \defeq d-\frac{b_nc_k}{a_m}.
\end{align}
\begin{remark}\label{S-non-elliptic}
(i) It is obvious from \eqref{leading-coefficient-Schur} that the ellipticity of $S(\la)$ breaks down whenever $\la\in\CC\setminus{\cl}\{d(\RR)\}$ lies in the range of $\Delta$. In Section~\ref{sec:DN-non-ellipticity} we will show that such points belong to the essential spectrum of $T$.

\smallskip
\noindent
(ii) The conditions $m=n+k$ and $D_0$ is of zero order are due to the employment of a result from \cite{ALMS94} on the essential spectrum of matrix differential operators generated by \eqref{A0} over compact intervals. We assume $m$ to be even because of the estimate for the numerical range of the Schur complement (see Lemma~\ref{lem:num.range}) which in turn is crucial for the main result of the paper.
\end{remark}

\section{Essential spectrum due to the singularity at infinity}\label{sec:ess.spec.due.infty}

In this section we are concerned with the description of the part of the essential spectrum that arises because of the singularity at infinity. Here we need the uniform ellipticity of the Schur complement $S(\la)$ and hence we exclude all $\la\in{\cl}\{\Delta(\RR)\}$, see Remark~\ref{S-non-elliptic}(i). 

The method to be used in this section is identical to the one of \cite{IT-2016-PsiDO} and the results are more or less particular cases of the corresponding results therein, although in \cite{IT-2016-PsiDO} both of the diagonal entries are positive-order pseudo-differential operators. Nevertheless, we provide here full details in order to make the paper self-contained. 

The hypotheses on the coefficient functions of the first Schur complement as well as of some auxiliary operators (Assumptions $(\rm{B}1)$-$(\rm{B}3)$ below) rule out the use of classical pseudo-differential operator theory.

\begin{definition}
For $\ell\in\RR$, the H\"ormander symbol class\footnote{also called the Kohn-Nirenberg symbol class, see e.g.\ \cite{Taylor-81b}.} $S^\ell(\RR^2)$ is defined to be the set of all infinitely smooth functions $\sigma:\RR^2 \to \CC$ such that for all $\alpha,\beta\in\NN_0$, there is a positive constant $C_{\alpha,\beta}$, depending only on $\alpha,\beta$, for which 
\[
|(\partial^{\beta}_x\partial^{\alpha}_{\xi})(x,\xi)| \leq C_{\alpha,\beta}(1+|\xi|)^{\ell-\alpha}, \quad (x,\xi)\in\RR^2,
\]
see e.g.\ \cite{Taylor-81b}, \cite{Wong-14b}. Further, we set  
\[
S^{-\infty}(\RR^2) := \bigcap_{\ell\in\RR}S^\ell(\RR^2), \quad S^{\infty}(\RR^2) := \bigcup_{\ell\in\RR}S^\ell(\RR^2)
\]
and recall that for $\sigma\in S^{\infty}(\RR^2)$, the pseudo-differential operator $T_{\sigma}$ with symbol $\sigma$ on the Schwartz space $\mathscr{S}(\RR)$ is defined by
\[
(T_{\sigma}\phi)(x) = \frac{1}{\sqrt{2\pi}}\int_{\RR}{\rm{e}}^{\ii x\xi}\sigma(x,\xi)\wh{\phi}(\xi) {\rm{d}}\xi, \quad \phi\in\Dom(T_{\sigma})=\mathscr{S}(\RR),
\]
where $\wh{\phi}$ is the Fourier transform of $\phi\in\mathscr{S}(\RR)$,
\[
\wh{\phi}(\xi) = \frac{1}{\sqrt{2\pi}}\int_{\RR}{\rm{e}}^{-\ii x\xi}\phi(x) {\rm{d}}x, \quad \xi\in\RR.
\]
\end{definition}

\begin{ass}\label{ass:B} Suppose that, for every $\la\in\CC\setminus\bigl({\cl}\{\Delta(\RR)\}\cup \cl \{d(\RR)\}\bigr)$, 
	\noindent
	\begin{enumerate}[label={{\upshape(B\arabic{*})}}]
	\item \label{ass:Schur-1}
	$p_j(\mydot,\la)\in B^\infty(\RR,\CC)$, $j\in\{0,1,\ldots,m\}$;
	\smallskip
	\item \label{ass:Schur-2}	
	$\dfrac{1}{p_{m}(\mydot,\la)}$ is bounded on $\RR$;
	\smallskip
	\item \label{ass:Schur-3}
	for all $\gamma\in\{0,1,\ldots,k\}$ and $\beta\in\{0,1,\ldots,n\}$, 
	\begin{align}
	c_\gamma (d-\la)^{-1}, \quad \bigl(b_\beta(d-\la)^{-1}\bigr)^{(j)}\in B^\infty(\RR,\CC), \quad j=0,1,\ldots,\beta.
	\end{align}
	\end{enumerate}
\end{ass}
\begin{remark}\label{rem:Schur-psido}
It is easy to see from \eqref{leading-coefficient-Schur} that the assumption \ref{ass:Schur-2} is automatically satisfied if $d$ is bounded 
and $\inf\limits_{x\in\RR}|a_m(x)|>0$.
\end{remark}

Note that assumption \ref{ass:Schur-1} allows us to consider the first Schur complement $S_0(\la)$ as a pseudo-differential operator on $\sS(\RR)$ with symbol $\sigma_\la$, given by   
\begin{equation}\label{symbol-S}
\sigma_{\la}(x,\xi) := \sum_{j=0}^{m}p_j(x,\la)\xi^j, \quad (x,\xi)\in\RR^2,
\end{equation}
belonging to the symbol class $S^m(\RR^2)$. By Assumption~\ref{ass:Schur-2}, the corresponding minimal operator $S(\la)$ is uniformly elliptic and hence
\begin{align}\label{dom-of-S}
\Dom(S(\la))=H^m(\RR),
\end{align}
see e.g.\ \cite{Wong-14b}. Moreover, $S(\la)$ has a parametrix, i.e. there exists a pseudo-differential operator $S^{\rm{p}}(\la)$ with symbol in $S^{-m}(\RR^2)$ and pseudo-differential operators $L_\la$ and $R_\la$ with symbols in $S^{-\infty}(\RR^2)$ such that 
	\begin{align}
	\label{parametrix-for-S}
	S^{\rm{p}}(\la)S(\la) = I+L_\la\,, \quad S(\la)S^{\rm{p}}(\la) = I+R_\la,
	\end{align}
respectively on $H^m(\RR)$ and $L^2(\RR)$.

In view of Assumption~\ref{ass:Schur-3}, whenever $\la\notin {\cl}\{d(\RR)\}$, the differential expressions $(D_0-\la)^{-1}C_0$ and $B_0(D_0-\la)^{-1}$ induce pseudo-differential operators on the Schwartz space $\sS(\RR)$ with symbols respectively in $S^k(\RR^2)$ and $S^n(\RR^2)$. We will need the following extensions of these operators,
\begin{align}
\label{def:F1}
F_1(\la)& := (D_0-\la)^{-1}C_0, \qquad \Dom(F_1(\la)):=H^k(\RR),\\
\label{def:F2}
F_2(\la)& := B_0(D_0-\la)^{-1}, \qquad \Dom(F_2(\la)):=H^n(\RR).
\end{align}

Furthermore, we will need the characterization of semi-Fredholm operators in terms of approximate inverses. Following \cite{EE87}, an operator $T\in\mathscr{C}(X)$ is said to have a \textit{left approximate inverse} if, and only if, there are operators $R_{\ell}\in\mathscr{B}(X)$ and $K_X\in\mathscr{K}(X)$ such that $I_X+K_X$ extends\footnote{\,Here it is sufficient to verify the equality $I_X+K_X=R_{\ell}T$ on any core of $T$.} $R_{\ell}T$. 

We need the following fact, the proof of which can be easily read off from the proofs of \cite[Theorems I.3.12-13]{EE87} and \cite[Lemma I.3.12]{EE87}.

\begin{proposition}\label{prop:left.apprx.inv}
	If $T\in\mathscr{C}(X)$ has a left approximate inverse, then $T$ has closed range and finite nullity.
\end{proposition}

It is a well-known fact that Fredholm operators admit two-sided approximate inverses. We will also need a special two-sided approximate inverse which can be described as follows. Let $T\in\sC(X)$ be Fredholm operator and define $\wt{T}$ to be the restriction of $T$ to $\Dom(T) \cap \Ker(T)^{\bot}$. Then $\wt{T}$ is injective and $\Ran(\wt{T})=\Ran(T)$ is closed. Hence the operator $\wt{T}^{-1}$, considered as a map from $\Ran(T)$ onto $\Dom(T) \cap \Ker(T)^{\bot}$, is bounded. Let $P$ and $I-Q$ be the orthogonal projections respectively onto $\Ker(T)$ and $\Ran(T)$. Defining the operator 
\begin{align}\label{def:gen.inv}
T^\dagger:=\wt{T}^{-1}(I-Q), \quad \Dom(T^\dagger):=X,
\end{align}
we immediately obtain
\begin{align}
\label{gen.inv-1}
T^\dagger T = I-P, \quad  T T^\dagger = I-Q
\end{align}
on $\Dom(T)$ and $X$, respectively. The constructed operator $T^\dagger\in\mathscr{B}(X)$ is called a \textit{generalized inverse} of $T$, see e.g.\ \cite{GGK-90b1}. Note that the generalized inverse is a two-sided approximate inverse since the operators $P$ and $Q$ are of finite-rank.

In contrast to the case when the underlying domain is compact, we can't view parametrices of pseudo-differential operators on the real line as two-sided approximate inverses. This is because pseudo-differential operators of negative orders on the real line are in general not compact in Sobolev spaces, see e.g. \cite[Section 2.3]{Agranovich90}. However, every parametrix $S^{\rm{p}}(\la)$ of $S(\la)$ is connected to its generalized inverse $S^\dagger(\la)$ by the following simple relationship which plays a crucial role in this paper.

\begin{lemma}\label{lem:param-vs-apprx.inv}
Let Assumptions~\ref{ass:A}, \ref{ass:B} be satisfied and $\la \notin {\cl}\{\Delta(\RR)\}\cup{\cl}\{d(\RR)\}$ be such that $S(\la)$ is Fredholm. Let $S^\dagger(\la)$ and $S^{\rm{p}}(\la)$ be respectively an approximate inverse and a parametrix of $S(\la)$. Then
	\begin{align}
	\label{rep:param-vs-apprx.inv}
	S^\dagger(\la)= S^{\rm{p}}(\la)(I-Q_\la) + L_\la S^\dagger(\la) R_\la - L_\la(I-P_\la)S^{\rm{p}}(\la),
	\end{align}
on $L^2(\RR)$, where $P_\la$ and $I-Q_\la$ are the orthogonal projections onto $\Ker(S(\la))$ and $\Ran(S(\la))$, respectively.
\end{lemma}
\begin{proof}
By the first relations in \eqref{parametrix-for-S} and \eqref{gen.inv-1}, we get 
	\begin{equation}\label{gen.inv-11}
	(S^\dagger(\la)-S^{\rm{p}}(\la))S(\la) = -P_\la-L_\la 
	\end{equation}
on $\Dom(S(\la))=H^m(\RR)$. Similarly, the second relations in \eqref{parametrix-for-S} and \eqref{gen.inv-1} yield 
	\begin{equation}\label{gen.inv-22}
	S(\la) (S^\dagger(\la)-S^{\rm{p}}(\la)) = -Q_\la-R_\la
	\end{equation}
on $L^2(\RR)$. Hence, multiplying both sides of \eqref{gen.inv-11} by $S^\dagger(\la)$ from the right and using the second relation in \eqref{gen.inv-1}, we obtain    
	\begin{align}
	\label{T:rep-1}
	(S^\dagger(\la)-S^{\rm{p}}(\la))(I-Q_\la) = -P_\la S^\dagger(\la) -L_\la S^\dagger(\la) 
	\end{align}
on $L^2(\RR)$. Similarly, multiplying both sides of \eqref{gen.inv-22} by $S^\dagger(\la)$ from the left and using the first relation in \eqref{gen.inv-1}, we obtain 
	\begin{align}
	\label{T:rep-2}
	S^\dagger(\la) = (I-P_\la)S^{\rm{p}}(\la)-S^\dagger(\la) R_\la + P_\la S^\dagger(\la)-S^\dagger(\la) Q_\la 
	\end{align}
on $L^2(\RR)$. Observe from the definition of $S^\dagger(\la)$ that $P_\la S^\dagger(\la) = S^\dagger(\la) Q_\la=0$ since $P_\la$ and $Q_\la$ are projections, see \eqref{def:gen.inv}. The claim thus follows by inserting \eqref{T:rep-2} for $S^\dagger(\la)$ on the right-hand side of \eqref{T:rep-1}.
\end{proof}

In the sequel, we will need the following lemma and its important corollary. 

\begin{lemma}\label{lem:kernel-smooth}
Let Assumptions~\ref{ass:A}, \ref{ass:B} be satisfied and $\la \notin {\cl}\{\Delta(\RR)\}\cup{\cl}\{d(\RR)\}$. Then
	\begin{enumerate}[font=\upshape,label=\upshape(\roman*\upshape)]
	\item $(\phi_1, \phi_2)^t \!\in\! \Ker(\cA-\la)$ implies $\phi_1 \in \Ker(S(\la))$ and $\phi_2=-(D_0-\la)^{-1}C_0\phi_1$;
	\item $(\phi_1, \phi_2)^t \!\in\! \Ker(\cA^*-\ov\la)$ implies $\phi_1 \in \Ker(S(\la)^*)$ and $\phi_2=-  (D_0^*-\ov{\la})^{-1}B_0^*\phi_1$.
	\end{enumerate}
Furthermore, in either case, we have  
	\begin{align}\label{H-alpha-1}
	(\phi_1, \phi_2)^t \!\in\! \sS(\RR) \!\oplus\! \sS(\RR).
	\end{align}      
\end{lemma}
\begin{proof}
We give the proof of {\rm{\eqref{H-alpha-1}}} and the claim in (i) only; the claim in (ii) can be proven in the same way. Let $\Phi := (\phi_1, \phi_2)^t\in \Ker(\cA-\la)$ be arbitrary. Then, for all $\Psi:=(\psi_1,\psi_2)\in C_0^\infty(\RR)\oplus C_0^\infty(\RR)$, we have $\langle (\cA-\la)\Phi, \Psi\rangle=0$. Since $\Psi\in\Dom(\cA_0^*)$ it therefore follows that $\langle\Phi,(\cA_0^*-\ov{\la})\Psi\rangle=0$ or, equivalently,
	\begin{align}\label{id:ker-1}
	\langle\phi_1, (A_0^*-\ov\la)\psi_1+C_0^*\psi_2\rangle + \langle\phi_2, B_0^*\psi_1+(D_0^*-\ov\la)\psi_2\rangle = 0.
	\end{align}
Setting $\psi_2 \!=\! -(D_0^*-\ov{\la})^{-1}B_0^*\psi_1$ for $\psi_1\in C_0^\infty(\RR)$ in \eqref{id:ker-1}, we get $\langle\phi_1, S_0(\la)^*\psi_1\rangle \!=\!0$ for all $\psi_1\in C_0^\infty(\RR)$. 
Therefore, $\phi_1\in \Dom(S(\la))=H^m(\RR)$ and $S(\la)\phi_1=0$. Consequently, using the first relation in \eqref{parametrix-for-S}, we obtain $	(I+L_\la\,)\phi_1=S^{\rm{p}}(\la)S(\la)\phi_1=0$ or $\phi_1=-L_\la\,\phi_1$. Since $L_\la\,$ is a pseudo-differential operator with symbol from $S^{-\infty}(\RR^2)$, it follows that $\phi_1\in H^{\alpha}(\RR)$ for every $\alpha\in\RR$, see e.g.\ \cite{Wong-14b}. Hence 
\begin{equation}\label{phi1-in-Schwartz}
\phi_1\in\bigcap_{\alpha\in\RR}H^{\alpha}(\RR)=\sS(\RR).
\end{equation}

On the other hand, setting $\psi_1=0$ in \eqref{id:ker-1}, we obtain 
	\begin{align}\label{id:ker-2}
	\langle\phi_1, C_0^*\psi_2\rangle + \langle\phi_2, (D_0^*-\ov\la)\psi_2\rangle = 0
	\end{align}
for all $\psi_2\in C_0^\infty(\RR)$. Hence for $\psi_2=(D_0^*-\ov{\la})^{-1}\varphi$ with arbitrary $\varphi\in C_0^\infty(\RR)$, 
	\begin{align}\label{id:ker-3}
	\langle\phi_1, C_0^*(D_0^*-\ov{\la})^{-1} \varphi\rangle + \langle\phi_2, \varphi\rangle = 0.
	\end{align}
Because of $\phi_1 \in \sS(\RR)$ by \eqref{phi1-in-Schwartz}, we have $\langle (D_0-\la)^{-1}C_0\phi_1+\phi_2, \varphi\rangle=0$ and therefore the density of $C_0^\infty(\RR)$ in $L^2(\RR)$ yields $\phi_2=-(D_0-\la)^{-1}C_0\phi_1$.

Now \eqref{H-alpha-1} is easily seen to hold as $(D_0-\la)^{-1}C_0$ is a differential operator of order $k$ and hence also $\phi_2=-(D_0-\la)^{-1}C_0\phi_1\in\sS(\RR)$.
\end{proof}

\begin{corollary}\label{cor:kernel-smooth}
	Under the assumptions of Lemma~\rm{\ref{lem:kernel-smooth}}, we have
	\begin{align}\label{dimA-dimS}
	\df(\cA-\la)<\infty \quad  \Longleftrightarrow \quad \df(S(\la))<\infty.
	\end{align}
\end{corollary}
\begin{proof}
``$\Longrightarrow$" in \eqref{dimA-dimS}: Suppose that $\df(\cA-\la)=N<\infty$ and let $\{\phi_n\}^N_{n=1}$, $\phi_n:=(\phi_{n,1}, \phi_{n,2})^t$, be an orthonormal basis for $\Ran(\cA-\la)^\bot=\Ker(\cA^*-\ov\la)$. Then by Lemma~\ref{lem:kernel-smooth} (ii), we have $\phi_{n,1} \in \Ker(S(\la)^*)$ and $\phi_{n,2}=-(D_0^*-\ov{\la})^{-1}B_0^*\phi_{n,1}$ for all $n\in\{1,\ldots,N\}$. If $\sum_{n=1}^Nc_n\phi_{n,1}=0$ for some constants $c_1,\ldots,c_N$, then 
	\begin{align*}
	\sum_{n=1}^Nc_n\phi_{n,2}=-\sum_{n=1}^Nc_n (D_0^*-\ov{\la})^{-1}B_0^*\phi_{n,1}=-(D_0^*-\ov{\la})^{-1}B_0^*\Bigl(\sum_{n=1}^Nc_n\phi_{n,1}\Bigr)=0
	\end{align*}
and hence $\sum_{n=1}^Nc_n\phi_n=0$. Since $\{\phi_n\}_{n=1}^N$ is a basis, we get $c_n=0$, $n\in\{1,\ldots,N\}$. Hence $\{\phi_{n,1}\}_{n=1}^N$ are linearly independent and thus $\dim\Ker(S(\la)^*)\geq N$. 
	
If $\dim\Ker(S(\la)^*) \geq N+1$, then there would exist $\phi_{0,1}\in\Ker(S(\la)^*)$ such that 
$\{\phi_{n,1}\}^N_{n=0}$ are linearly independent. Since $\phi_{0,1}\in\Ker(S(\la)^*)$, it would then follow that $\phi_{0,1}\in \sS(\RR)$, $\phi_0:=(\phi_{0,1}, -(D_0^*-\ov{\la})^{-1}B_0^*\phi_{0,1})^t \in \Ker(\cA^*-\ov{\la})$. Hence $\{\phi_n\}^N_{n=0}$ 
would be linearly independent, contradicting 
$\df(\cA-\la)=N$. Therefore, 	
\[
\df(S(\la))=\dim\Ker(S(\la)^*) = N.
\]
	
\smallskip
\noindent
``$\Longleftarrow$" in \eqref{dimA-dimS}: Suppose that $\df(S(\la)) \!=\! N<\infty$ and let $\{\phi_{n,1}\}_{n=1}^N$ be an orthonormal basis for $\Ran(S(\la))^\bot \!=\! \Ker(S(\la)^*)$. Consider $\{\phi_n\}^N_{n=1} \!\in\! \Ker(\cA^*-\ov{\la})$, where $\phi_n:=(\phi_{n,1}, \phi_{n,2})^t$ with $\phi_{n,2}:=-(D_0^*-\ov{\la})^{-1}B_0^*\phi_{n,1}$. These vectors must be linearly independent, for otherwise $\{\phi_{n,1}\}_{n=1}^{N}$ would be linearly dependent, contradicting  $\df(S(\la))=N$. Hence $\dim\Ker(\cA^*-\ov{\la}) \geq N$. 
 
If $\dim\Ker(\cA^*-\ov{\la}) \!\geq\! N+1$, then there would exist $\phi_0:=(\phi_{0,1},\phi_{0,2})^{t} \!\in\! \Ker(\cA^*-\ov{\la})$ such that $\{\phi_n\}^N_{n=0}$ are linearly independent. It would then follow from Lemma~\ref{lem:kernel-smooth} that $\phi_{0,1} \in \Ker(S(\la)^*)$ and $\phi_{0,2} = -(D_0^*-\ov{\la})^{-1}B_0^*\phi_{0,1}$. This would imply that $\{\phi_{n,1}\}_{n=0}^{N}$ are linearly independent and thus $\dim \Ker(S(\la)^*) \geq N+1$. This contradiction yields  
\[
\df(\cA-\la)=\dim\Ker(\cA^*-\ov{\la})=N.\qedhere
\]	
\end{proof}

A key result in the description of the essential spectrum due to the singularity at infinity is the following characterization in terms of the essential spectrum of the Schur complement.

\begin{theorem}\label{thm:Schur-complement}
Let Assumptions~\ref{ass:A}, \ref{ass:B} be satisfied. Then 
	\begin{equation}\label{Schur-equiv}
	\la\in\sess(\cA) \quad \Longleftrightarrow \quad 0\in\sess(S(\la)),
	\end{equation} 
provided that $\la\notin{\cl}\{\Delta(\RR)\}\cup{\cl}\{d(\RR)\}$. 
\end{theorem}

\begin{proof} 
Let $\la \notin {\cl}\{\Delta(\RR)\}\cup{\cl}\{d(\RR)\}$ be fixed. In view of Proposition~\ref{prop:left.apprx.inv} and Corollary~\ref{cor:kernel-smooth}, it suffices to show that if one of $\cA-\la$ and $S(\la)$ is Fredholm, then the other has a left approximate inverse. 
	
First assume that $\cA-\la$ is Fredholm. Take an arbitrary $u\in\sS(\RR)$ and define $v:=-(D_0-\la)^{-1}C_0u \in \sS(\RR)$. Then, by the definition of $v$, we have
	\begin{equation}\label{lemschur:<=1}
	(\cA -\la) 
	\begin{pmatrix}
	u \\ v
	\end{pmatrix}
	=
	\begin{pmatrix}
	S(\la)u \\ 0
	\end{pmatrix}.
	\end{equation}
Let $\cA^\dagger(\la)$ be the generalized inverse of $\cA-\la$. Then $\cA^\dagger(\la)(\cA-\la)=I-\wh P_\la$ on $\Dom(\cA)$, where $\wh P_\la$ is the orthogonal projection onto $\Ker(\cA-\la)$, see \eqref{gen.inv-1}. Hence \eqref{lemschur:<=1} yields
	\begin{equation}\label{toshmat2}
	\begin{pmatrix}
	u \\ v
	\end{pmatrix}
	-\wh P_\la
	\begin{pmatrix}
	u \\ v
	\end{pmatrix}
	=\cA^\dagger(\la)(\cA-\la)
	\begin{pmatrix}
	u \\ v
	\end{pmatrix}
	=\cA^\dagger(\la)
	\begin{pmatrix}
	S(\la)u \\ 0
	\end{pmatrix}.
	\end{equation}
Note that $k:=\dim\Ker(\cA-\la)<\infty$ since $\cA-\la$ is Fredholm. Let $\{(f_j,g_j)^t\}_{j=1}^k$ be an orthonormal basis for $\Ker(\cA-\la)$. By \eqref{H-alpha-1}, we have 
\[
\varphi_j := C_0^*(D_0^*-\ov{\la})^{-1} g_j-f_j\in L^2(\RR),\quad j=1,2,\ldots,k.
\]
Hence the operator $\wh{K}(\la):L^2(\RR)\to L^2(\RR)$, defined by
\begin{align*}
\wh{K}(\la)f:=\sum_{j=1}^k\langle f,\varphi_j\rangle f_j, \quad f\in\Dom(\wh K_\la)=L^2(\RR),
\end{align*}	
is compact. Denoting by $P_1:L^2(\RR)\oplus L^2(\RR) \to L^2(\RR)$ the projection onto the first component and defining the bounded operator $S_\ell(\la):L^2(\RR)\to L^2(\RR)$ as
	\[
S_{\ell}(\la)f:=P_1\cA^\dagger(\la)\binom{f}{0}, \quad f\in\Dom(S_\ell(\la))=L^2(\RR),
	\]
one easily obtains from \eqref{toshmat2} that
	\[
	u+\wh{K}(\la)u=S_{\ell}(\la)S(\la)u.
	\]
Since $u\in \sS(\RR)$ was arbitrary and $\sS(\RR)$ is a core for $S(\la)$, it follows that $S_\ell(\la)$ is a left approximate inverse for $S(\la)$.

Now assume that $S(\la)$ is Fredholm. For $(f,g)^{t} \!\in\! \Ran(\cA_{0}-\la)\subset C_0^\infty(\RR)\oplus C_0^\infty(\RR)$, we have 
	\begin{align}\label{toshmat3}
	\begin{pmatrix}
	f\\g
	\end{pmatrix}
	=(\cA_{0}-\la) \begin{pmatrix}
	u\\v
	\end{pmatrix}
	\quad \Longleftrightarrow \quad 
	\begin{matrix}
	(A_0-\la)u+B_0v=f,\\
	C_0u+(D_0-\la)v=g,
	\end{matrix}
	\end{align}
with $(u,v)^{t}\in\Dom(\cA_{0})\subset C_0^\infty(\RR)\oplus C_0^\infty(\RR)$, and it follows that 
	\begin{equation}
	\label{Schur-u}
	S(\la)u=f-B_0(D_0-\la)^{-1}g=f-F_2(\la)g.
	\end{equation}
Let $S^\dagger(\la)$ be the generalized inverse of $S(\la)$. Applying $S^\dagger(\la)$ to \eqref{Schur-u} and using \eqref{gen.inv-1} we find
	\begin{equation}
	u=S^\dagger(\la)\,f-S^\dagger(\la)\,F_2(\la)g+\,P_\la u,
	\end{equation}
where $P_\la$ is the orthogonal projection onto $\Ker(S(\la))$, see \eqref{gen.inv-1}. Inserting this into the last equation in \eqref{toshmat3} and solving for $v$ we obtain
	\begin{align*}
	v = & \,(D_0-\la)^{-1}g-F_1(\la)S^\dagger(\la)\,f +F_1(\la)S^\dagger(\la)\,F_2(\la)g-F_1(\la)\,P_\la u.
	\end{align*}
Therefore,
	\begin{align}\label{apprx.inv-1}
	\begin{pmatrix}
	u\\v
	\end{pmatrix}
	=\cA^{\#}(\la) \begin{pmatrix}
	f\\g
	\end{pmatrix}
	-K(\la)  \begin{pmatrix}
	u\\v
	\end{pmatrix},
	\end{align}
where
	\begin{align}
	\cA^{\#}(\la) := \begin{pmatrix}
	S^\dagger(\la)\, & \quad -S^\dagger(\la)\,F_2(\la)\\[2ex]
	-F_1(\la)S^\dagger(\la)\, & \quad (D_0-\la)^{-1} + F_1(\la)S^\dagger(\la)\,F_2(\la)
	\end{pmatrix}, 
	\end{align}
with domain
		\[
		\quad \Dom(\cA^{\#}(\la)) := L^2(\RR)\oplus H^n(\RR),
		\]
and 
	\begin{align}
	K(\la) := \begin{pmatrix}
	-P_\la & \: 0\\[2ex]
	F_1(\la)P_\la & \: 0
	\end{pmatrix}, \quad \Dom(K(\la)):=L^2(\RR)\oplus L^2(\RR).
	\end{align}
It is not difficult to see that the operators $\cA^{\#}(\la)$ and $K(\la)$ are well-defined since both $\Ran(P_\la)$ and $\Ran(S^\dagger(\la))$ are subsets of $H^m(\RR)$ and the latter is contained in $H^s(\RR)=\Dom(F_1(\la)) \cap \Dom(F_2(\la))$ where $s=\min\{n,k\}\leq m$. Moreover, $K(\la)$ is a compact operator in $L^2(\RR)\oplus L^2(\RR)$ since $P_\la:L^2(\RR)\to L^2(\RR)$ is a finite-rank operator with $\Ran(P_\la)\subset\Ker(S(\la)) \subset \Dom(F_1(\la))$.
	
Since $\sS(\RR)\oplus \sS(\RR)$ is a core for $\cA-\la$, it is left to be shown that $\cA^{\#}(\la)$ has a bounded extension to $L^2(\RR)\oplus L^2(\RR)$. To this end, observe that with the help of Lemma~\ref{lem:param-vs-apprx.inv}, we have the decomposition $\cA^{\#}(\la)=\cA_1^{\#}(\la)+\cA_2^{\#}(\la)$ where
	\begin{align}
	\cA_1^{\#}(\la) :=& \begin{pmatrix}
	S^{\rm{p}}(\la) & \quad -S^{\rm{p}}(\la)F_2(\la)\\[2ex]
	-F_1(\la)S^{\rm{p}}(\la) & \quad (D_0-\la)^{-1} + F_1(\la)S^{\rm{p}}(\la)F_2(\la)
	\end{pmatrix},
	\end{align}
and
\begin{align}
	\cA_2^{\#}(\la) :=& \begin{pmatrix}
	H(\la) & \quad -H(\la)F_2(\la)\\[2ex]
	-F_1(\la)H(\la) & \quad  F_1(\la)H(\la)F_2(\la)
	\end{pmatrix},
	\end{align}
with 
\[
\Dom(\cA_1^{\#}(\la))=\Dom(\cA_2^{\#}(\la))=L^2(\RR)\oplus \Dom(F_2(\la))
\]
and
	\begin{align*}
	H(\la) := L_\la S^\dagger(\la)R_\la-L_\la(I-P_\la)S^{\rm{p}}(\la)-S^{\rm{p}}(\la)Q_\la.
	\end{align*}
First we justify that $\cA_1^{\#}(\la)$ has a bounded extension to $L^2(\RR)\oplus L^2(\RR)$. By \cite[Theorem 8.1]{Wong-14b}, $S^{\rm{p}}(\la)F_2(\la)$, $F_1(\la)S^{\rm{p}}(\la)$ and  $F_1(\la)S^{\rm{p}}(\la)F_2(\la)$ are pseudo-differential operators with symbols from $S^{-k}(\RR^2)$, $S^{-n}(\RR^2)$ and $S^0(\RR^2)$, respectively. Hence, by \cite[Theorem 12.9]{Wong-14b}, these operators have bounded extensions to $L^2(\RR)$ or are bounded in $L^2(\RR)$. Since $S^{\rm{p}}(\la)$ is bounded in $L^2(\RR)$, it follows that all the entries of $\cA_1^{\#}(\la)$ have bounded extensions to $L^2(\RR)$.  
	
The existence of a bounded extension of $\cA_2^{\#}(\la)$ to $L^2(\RR)\oplus L^2(\RR)$ can be shown similarly. Indeed, it is easy to see that $H(\la)$ has a bounded extension to $L^2(\RR)$. Furthermore, because $F_1(\la)L_\la$ and $F_1(\la)S^{\rm{p}}(\la)$ are pseudo-differential operators with symbols respectively in $S^{-\infty}(\RR^2)$ and $S^{-n}(\RR^2)$, it follows that $F_1(\la)H(\la)$ is a bounded operator on $L^2(\RR)$. In the same way, it follows that $H(\la)F_2(\la)$ has a bounded extension to $L^2(\RR)$. Finally, by the above observations and also noting that $R_\la\,F_2(\la)$ is a pseudo-differential operator with symbol in $S^{-\infty}(\RR^2)$, we conclude that the operator $F_1(\la)H(\la)F_2(\la)$ has a bounded extension to $L^2(\RR)$. Therefore, \cite[Theorem 12.9]{Wong-14b} again implies that $\cA_2^{\#}(\la)$ has a bounded extension to $L^2(\RR)\oplus L^2(\RR)$.
\end{proof}

\section{Essential spectrum due to the non-ellipticity \\in the sense of Douglis and Nirenberg}
\label{sec:DN-non-ellipticity}

For $\la\in\CC$, the operator matrix $\cA-\la$ is \emph{elliptic in sense of Douglis and Nirenberg} on the real line if and only if 
\[
	\det M_\la(x,\xi)\neq0, \quad x\in\RR, \: \xi\neq0,
\]
where $M_\la(x,\xi)$ is the principal symbol of $\cA-\la$, given by the matrix consisting of the principal symbols of the entries,
\begin{align}
	M_{\la}(x,\xi) :=
	\begin{pmatrix} 
		a_m\xi^m & b_n\xi^n \\[2ex]
		c_k\xi^k & d-\la 
	\end{pmatrix}, 
	\quad (x,\xi)\in\RR^2,
\end{align}
see e.g.\ \cite{Agranovich90}, \cite{Miranda-70b}. Observe that 
\begin{align}\label{det-M}
	\det M_\la(x,\xi) = a_m(x) \bigl(\Delta(x)-\la\bigr)\xi^m, \quad (x,\xi)\in\RR^2,
\end{align}
where the function $\Delta$ is given by \eqref{defDelta}.

Since $a_m \neq 0$ on $\RR$ by Assumption~\ref{ass:A}, it is clear from the relation in \eqref{det-M} that the ellipticity of $\cA-\la$ in sense of Douglis and Nirenberg is violated exactly for those $\la$ which lie in (the closure of) the range of the function $\Delta:\RR\to\CC$. Our goal in this section is to prove that such points belong to essential spectrum of $\cA$. 

\begin{theorem}\label{thm:reg.part}
	Let Assumption \ref{ass:A} be satisfied. Then 
	\begin{equation*}
	\{\la\in\CC:\, \cA-\la \; \textrm{is not Douglis-Nirenberg elliptic}\}
	\subset \sess (\cA). 
	\end{equation*}
\end{theorem}

\begin{remark}
The proof of Theorem~\ref{thm:reg.part} is given below. Its analog was proven in \cite{ILLT13} in the symmetric case with $m=2$, $n=k=1$. The main tool for this was Glazman's decomposition principle combined with the result of \cite{ALMS94}. Recall that, by \cite{ALMS94}, if the compact interval $[0,1]$ is considered instead of $\RR$, then the essential spectrum of the closed operator generated by $\cA_0$ on a ``nice domain" over $[0,1]$ is given by $\Delta([0,1])$; of course the same holds for any compact interval $[a,b]$. Since our operator matrix is non-symmetric, this approach does not readily generalize to prove Theorem~\ref{thm:reg.part} as it is not obvious (in fact, it is a difficult problem) whether the deficiency indices of the corresponding minimal operator are finite. 	
\end{remark}

\subsection{Essential spectrum when the underlying domain is a compact interval}\label{subsec:preliminaries}
For given $a,b\in\RR$, we denote by $A$ the restriction of the differential expression $\tau_{11}(\mydot,D)$ to the domain determined by general boundary conditions
	\begin{align}\label{domain-of-A}
	\Dom(A):=\bigl\{y_1\in H^m(a,b):\, U(y_1)=0\bigr\}
	\end{align}
where 
	\begin{align*}
	U(y_1):=U_0\begin{pmatrix}
		y_1(a)\\y'_1(a)\\ \vdots\\y^{(m-1)}(a)
	\end{pmatrix}
	+U_1\begin{pmatrix}
		y_1(b)\\y'_1(b)\\ \vdots\\y^{(m-1)}(b)
	\end{pmatrix}
	\end{align*}
with $m\times m$ complex matrices $U_0, U_1\in M_{m}(\CC)$. We assume that the boundary-conditions are normalized and Birkhoff regular, see \cite{ALMS94} for more details. We denote by $B$ the restriction of the differential expression $\tau_{12}(\mydot,D)$ to the domain 
	\[
	\Dom(B):=\bigl\{y_2\in H^n(a,b):\, y_2^{(j)}(a)=y_2^{(j)}(b)=0,\,j=0,1,\ldots,n-1\bigr\}.
	\]
Furthermore, we denote by $C$, $D$ the restrictions of the differential expressions $\tau_{21}(\mydot,D)$, $\tau_{22}(\mydot,D)$ to the domains $\Dom(A)$, $\Dom(B)$, respectively. In the Hilbert space $L^2(a,b) \oplus L^2(a,b)$, we consider the operator matrix
	\[
	\text{L}_0:=\begin{pmatrix}
	A & B\\[1ex]
	C & D
	\end{pmatrix}, \qquad \Dom(\text{L}_0):=\Dom(A)\oplus\Dom(B).
	\]
Let $\cB_0$ denote the restriction of $\cA_0$ to $C_0^\infty(a,b)\oplus C_0^\infty(a,b)$. Since the domains of the adjoint operators to $\rm{L}_0$ and $\cB_0$ contain $C_0^\infty(a,b)\oplus C_0^\infty(a,b)$ and the latter is dense in $L^2(a,b)\oplus L^2(a,b)$, both $\rm{L}_0$ and $\cB_0$ are closable. Let $\rm{L}$ and $\cB$ denote the closures of $\rm{L_0}$ and $\cB_0$, respectively. Clearly, $\rm{L}$ is a closed extension of $\cB$. In fact, we have the following more precise result.

\begin{proposition}\label{prop:ALMS-O}
$\rm{L}$ is a finite-dimensional extension of $\cB$, that is,
	\begin{align}\label{prop:ALMS-0-find.dim}
	\dim \Dom(\rm{L}) / \Dom(\cB) <\infty.
	\end{align}
\end{proposition}
\begin{proof}
The operator $A$ with domain \eqref{domain-of-A} has compact resolvent, see \cite{ALMS94}. Consequently, $\sigma(A)$ consists exclusively of isolated eigenvalues of finite algebraic multiplicities. It is known from \cite[Theorem~4.3]{ALMS94} that, for $\la\in\rho(A)$, the operator 
	\[
	S_{2,0}(\la) := D-\la-C(A-\la)^{-1}B, \quad \Dom(S_{2,0}(\la)):=\Dom(B),
	\]
admits a bounded closure $S_2(\la):=\ov{S_{0,2}(\la)}$. From now on, we assume that $\la\in\rho(A)$ and $\mu\in\CC$ are chosen in such a way that 
\begin{align}
\la\notin\Delta([a,b]) \cup d([a,b]),
\end{align}
and 
\begin{align}\label{condition-on-la-mu}
\|S_2(\la)\| < |\mu|, \quad \mu+\la\notin d([a,b]);
\end{align}
such a choice of $\la$ and $\mu$ is possible because the sets $\Delta([a,b])$ and $d([a,b])$ correspond to curves of finite lengths in the complex plane. Consider the operators
	\begin{align*}
	\cB_{0,\mu} := \cB_0-
	 			\begin{pmatrix} 0 & 0\\[1ex]
	       		0 & \mu 
				\end{pmatrix}, 
				\quad 
	\rm{L}_{0,\mu} := \rm{L}_0-
					\begin{pmatrix} 0 & 0\\[1ex]
					0 & \mu  
					\end{pmatrix}
	\end{align*}
on the domains 
	\[
	\Dom(\cB_{0,\mu}) := \Dom(\cB_0), \quad \Dom(\rm{L}_{0,\mu}) := \Dom(\rm{L_0}).
	\]
Since closability and closedness are preserved under bounded perturbations, the closability of the operators $\cB_0$ and $\rm{L}_0$ implies that the operators $\cB_{0,\mu}$ and $\rm{L}_{0,\mu}$ are closable. Denoting their closures respectively by $\cB_\mu$ and $\rm{L}_\mu$, we have $\Dom(\cB_\mu)=\Dom(\cB)$ and $\Dom(\rm{L}_\mu)=\Dom(\rm{L})$. Hence \eqref{prop:ALMS-0-find.dim} is equivalent to
	\begin{align}\label{prop:ALMS-0-find.dim.1}
		\dim\, \Dom(\rm{L_\mu}) / \Dom(\cB_\mu) <\infty.
	\end{align}
By virtue of the Frobenius-Schur factorization of $\rm{L_{0,\mu}}$, we have 
	\begin{align}\label{Frobenius-Schur}
		\rm{L_\mu}-\la =\begin{pmatrix} 
			I & 0\\[2ex]
			G(\la) & I
		\end{pmatrix}
		\begin{pmatrix} 
			A-\la  & 0\\[2ex]
			0 & S_2(\la)-\mu 
		\end{pmatrix}
		\begin{pmatrix} 
			I & F(\la)\\[2ex]
			0 & I
		\end{pmatrix},
	\end{align}
where $G(\la):=C(A-\la)^{-1}$, which is everywhere defined and bounded (as a consequence of the closed graph theorem) operator in $L^2(\RR)$, and $F(\la)$ is the closure of $(A-\la)^{-1}B$.
	
Observe that, because $\la\in\rho(A)$ and $\mu\in\rho(S_2(\la))$, the middle term on the right-hand side of \eqref{Frobenius-Schur} is boundedly invertible. Clearly this is the case for the first and last terms as well. Hence $\la\in\rho(\rm{L_\mu})$. This observation, in particular, implies that $\la\in\Pi(\rm{L}_\mu)$.  Since $\cB_{\mu}$ is a restriction of $\rm{L}_\mu$, we get $\la\in\Pi(\cB_\mu)$. It is obvious that $\rm{L}_\mu-\la$ is Fredholm and hence the quantities $\df(\rm{L}_\mu-\la)$ and $\nl(\rm{L}_\mu-\la)$ are finite, the latter quantity being equal to zero. Therefore, if we show that 
	\begin{align}\label{prop:ALMS-0-find.dim.last.step}
		\df(\cB_\mu-\la)<\infty,
	\end{align}
then \cite[Theorem~III.3.1]{EE87} applies and gives \eqref{prop:ALMS-0-find.dim.1}, finishing the proof of the claim  in ~\eqref{prop:ALMS-0-find.dim}. 
	
Denote by $S_{0,\mu}(\la)$ the first Schur complement associated with the operator matrix $\cB_{0,\mu}-\la$. Since the domain of the adjoint of $S_{0,\mu}(\la)$ contains $C_0^\infty(a,b)$ and the latter is dense in $L^2(a,b)$, it follows that $S_{0,\mu}(\la)$ is closable; we denote its closure by $S_\mu(\la)$. 
Note that the coefficient functions of $S_{0,\mu}(\la)$ satisfy assumptions (i)-(iii) of \cite[p.445]{EE87} with $I=[a,b]$. Hence we have 
\begin{align}\label{dom-of-S-adjoint-[a,b]}
\Dom(S_\mu(\la)^*)=H^m(a,b),
\end{align}
see \cite[p.446]{EE87}. Furthermore, \cite[Lemma~IX.9.1]{EE87} yields $\sess(S_\mu(\la))=\emptyset$. In particular,  $\df(S_\mu(\la))<\infty$ and thus it suffices to prove that 
	\begin{align}\label{prop:ALMS-0-find.dim.last.step.1}
	\df(\cB_\mu-\la) \leq \df(S_\mu(\la)).
	\end{align}
To this end, let $\Phi=(\phi_1, \phi_2)^t\in\Ran(\cB_\mu-\la)^\bot = \Ker(\cB_\mu^*-\ov{\la})$ be arbitrary.
Then for all $\Psi=(\psi_1,\psi_2)^t\in C_0^\infty(a,b)\oplus C_0^\infty(a,b)$, we have $\langle \Psi, (\cB_\mu^*-\ov{\la})\Phi\rangle=0$. Since $\Psi\in\Dom(\cB_{0,\mu}-\la)$ it therefore follows that $\langle (\cB_{0,\mu}-\la)\Psi, \Phi\rangle=0$, or equivalently,
	\begin{align}\label{ALMS:id-ker-1}
		\langle (A_0-\la)\psi_1+B_0\psi_2, \phi_1 \rangle + \langle C_0\psi_1+(D_0-\mu-\la)\psi_2, \phi_2\rangle = 0.
	\end{align}
Setting $\psi_2 = -(D_0-\mu-\la)^{-1}C_0\psi_1$ with arbitrary $\psi_1\in C_0^\infty(a,b)$ in \eqref{ALMS:id-ker-1}, we obtain $\langle S_\mu(\la)\psi_1, \phi_1\rangle=0$ for all $\psi_1\in C_0^\infty(a,b)$. Hence $\phi_1\in \Dom(S_\mu(\la)^*)=H^m(a,b)$, see \eqref{dom-of-S-adjoint-[a,b]}, and $S_\mu(\la)^*\phi_1=0$, that is,  
	\begin{align}\label{ALMS:kernel-1}
	\phi_1\in \Ker(S_\mu(\la)^*).
	\end{align}
On the other hand, setting $\psi_1=0$ in \eqref{ALMS:id-ker-1}, we get for all $\psi_2\in C_0^\infty(a,b)$,
	\begin{align}\label{ALMS:id-ker-2}
	\langle B_0\psi_2, \phi_1\rangle + \langle (D_0-\mu-\la)\psi_2, \phi_2\rangle = 0.
	\end{align}
Letting $\psi_2 = (D_0-\mu-\la)^{-1}\varphi$ with arbitrary $\varphi\in C_0^\infty(a,b)$, we thus obtain
	\begin{align}\label{ALMS:id-ker-3}
	\langle B_0(D_0-\mu-\la)^{-1}\varphi, \phi_1 \rangle + \langle \varphi, \phi_2 \rangle = 0.
	\end{align}
Since $\phi_1 \in H^m(a,b) \subset \Dom((B_0(D_0-\mu-\la)^{-1})^*)$, we therefore have  
	\begin{align}\label{ALMS:id-ker-4}
	\langle \varphi, (B_0(D_0-\mu-\la)^{-1})^*\phi_1+\phi_2 \rangle=0.
	\end{align}
Since $\varphi\in C_0^\infty(a,b)$ was arbitrary, the density of $C_0^\infty(a,b)$ in $L^2(a,b)$ implies
	\begin{align}\label{ALMS:kernel-2}
	\phi_2 = -(B_0(D_0-\mu-\la)^{-1})^*\phi_1.
	\end{align}
Applying similar ideas as in the proof of Corollary~\ref{cor:kernel-smooth}, the claim in \eqref{prop:ALMS-0-find.dim.last.step.1} immediately follows from \eqref{ALMS:kernel-1} and \eqref{ALMS:kernel-2}.
\end{proof}

\begin{corollary}
	For the essential spectrum of $\cB$, we have
	\begin{equation}\label{ess.spec.cB}
	\sess(\cB)=\Delta([a,b]).
	\end{equation}
\end{corollary}
The proof of the claim \eqref{ess.spec.cB} is an immediate consequence of \cite[Theorem~4.5]{ALMS94} combined with \eqref{prop:ALMS-0-find.dim} and \cite[Corollary IX.4.2]{EE87}. 

\subsection{Numerical range of the first Schur complement}
In this subsection we establish a result on the numerical range of $S(\la)$. Here we need the following version of strong ellipticity. 

\begin{ass}\label{ass:C} 
Let Assumption~\ref{ass:B} be satisfied except for \ref{ass:Schur-2} being replaced by the condition that there exist constants $\theta_\la\in[0,\pi]$ and $\delta_\la>0$ such that 
	\begin{align}\label{strong-ellipticity}
	\Re({\rm{e}}^{\ii\theta_\la}p_{m}(x,\la)) \geq \delta_\la, \quad x\in\RR.
	\end{align}
\end{ass}

\begin{lemma}\label{lem:num.range}
Let Assumptions~\ref{ass:A}, \ref{ass:C} be satisfied and $\la\!\notin\!{\cl}\{\Delta(\RR)\} \cup {\cl}\{d(\RR)\}$. Then the closure of the numerical range $W(S(\la))$ of $S(\la)$ is contained in a sector of the complex plane with a semi-angle $\vartheta\in[0,\frac{\pi}{2})$. Moreover,
	\[
	\CC\setminus {\cl}\{W(S(\la))\} \subset \rho(S(\la)).
	\]
\end{lemma}
\begin{proof}
Let $m_1=m/2$ and take an arbitrary $\mu\in W(S_{0}(\la))$. Then $\mu=\langle S_{0}(\la)\phi,\phi\rangle$ for some $\phi\in\Dom(S_{0}(\la))=\sS(\RR)$ with $\|\phi\|_{L^2(\RR)}=1$ and hence
	\begin{align} 
	\nonumber
	\mu = \sum_{j=0}^m \bigl\langle (-\ii)^j p_j(\mydot,\la) \phi^{(j)},\phi\rangle = &\sum_{j=m_1}^m(-1)^{m_1}\bigl\langle (-\ii)^j\phi^{(j-m_1)}, (\ov{p_j(\mydot,\la)}\phi)^{(m_1)}\bigr\rangle \\\label{mu=}
	 &+ \sum_{j=0}^{m_1-1}\bigl\langle(-\ii)^j p_j(\mydot,\la)\phi^{(j)},\phi\bigr\rangle.
	\end{align}
By Erhling's inequality (see e.g.\ \cite{Wong-14b}), for any $\varepsilon_0>0$ there is a constant $K_{\varepsilon_0}>0$ depending on $\varepsilon_0$ only and such that
	\begin{equation}
	\|\phi^{(j)}\|_{L^2(\RR)} \leq \varepsilon_0\|\phi^{(m_1)}\|_{L^2(\RR)} + K_{\varepsilon_0}\|\phi\|_{L^2(\RR)}, \quad j=0,1,\ldots, m_1-1.  
	\end{equation}
Using these estimates together with the Cauchy-Schwartz inequality in \eqref{mu=} shows that for any $\varepsilon>0$, 
	\begin{align*}
	\bigl|{\rm{e}}^{\ii\theta_\la}\mu \clr{-} \bigl\langle\phi^{(m_1)}, {\rm{e}}^{-\ii\theta_\la}\ov{p_m(\mydot,\la)}\phi^{(m_1)}\bigr\rangle\bigr| 
	 \leq \varepsilon\|\phi^{(m_1)}\|^2_{L^2(\RR)}+K_{\varepsilon}\|\phi\|^2_{L^2(\RR)}
	\end{align*}
for some constant $K_{\varepsilon}>0$. Here it is taken into account that the functions $p_j(\mydot,\la)$, $j\in\{0,1,\ldots,m\}$ and their derivatives up to order $m_1$ are bounded on $\RR$. Therefore, choosing $\varepsilon=\delta_\la/2$ with $\delta_\la$ as in Assumption~\ref{ass:C}, we obtain
	\begin{align*}
	\Re({\rm{e}}^{\ii\theta_\la}\mu)+K_{\delta_\la} &\geq \int_{\RR}\Bigl(\Re({\rm{e}}^{\ii\theta_\la}p_m(\mydot,\la))-\frac{\delta_\la}{2}\Bigr)|\phi^{(m_1)}|^2 \, {\rm{d}}x \\
	&\geq \frac{\delta_\la}{2}\|\phi^{(m_1)}\|^2_{L^2(\RR)}\geq0.
	\end{align*}
Furthermore, for some positive constant $\gamma_\la$,
	\begin{align}
	|\Im({\rm{e}}^{\ii\theta_\la}\mu)| \leq \gamma_\la\bigl(\|\phi^{(m_1)}\|^2_{L^2(\RR)}+\|\phi\|^2_{L^2(\RR)}\bigr) \leq \frac{2\gamma_\la}{\delta_\la}\bigl(\Re({\rm{e}}^{\ii\theta_\la}\mu)+K_{\delta_\la}\bigr).
	\end{align}
This shows that $S_{0}(\la)$ is sectorial with $W(S_{0}(\la))$ contained in a sector with semi-angle $\vartheta\in[0,\frac{\pi}{2})$ from which the first claim follows.

Next, let $\omega\notin{\cl}\{W(S(\la))\}$ be arbitrary. By \cite[Theorem~III.2.3]{EE87}, we know that $\nl(S(\la)-\omega)=0$ and $\Ran(S(\la)-\omega)$ is closed. Note that $\mu\in W(S(\la))$ if and only if $\ov{\mu}\in W(S(\la)^*)$. So $\ov{\omega}\notin {\cl}\{W(S(\la)^*)\}$ and, consequently, 
\[
\df(S(\la)-\omega I)=\nl(S(\la)^*-\ov{\omega}I)=0.
\]
Therefore, any $\omega\notin{\cl}\{W(S(\la))\}$ lies in the resolvent set $\rho(S(\la))$.
\end{proof}

\subsection{Some more auxiliary results}\label{sub.sec:mor.aux.results}

For given $a,b\in\RR$, $a<b$ let $I_1=(-\infty,a]$, $I_2=[a,b]$ and $I_3=[b,\infty)$. For each $j=1,2,3$, denote by $\cB_{I_j}$ and $S_{I_j}(\la)$ respectively the closure of the restriction of the operator matrix $\cA_0$ to 
$C_0^\infty(I_j)\oplus C_0^\infty(I_j)$ and the closure of the corresponding first Schur complement (the closability of these restrictions can be easily seen from the fact that the corresponding adjoint operators have dense domains). 
Then $\cA$ is a closed extension of the orthogonal sum 
\begin{equation}
T:=\cB_{I_1} \oplus \cB_{I_2} \oplus \cB_{I_3}, \quad \Dom(T):=\Dom(\cB_{I_1}) \oplus \Dom(\cB_{I_2}) \oplus \Dom(\cB_{I_3}).
\end{equation}

\begin{lemma}\label{lem:reg.field}
Let Assumptions~\ref{ass:A}, \ref{ass:B} be satisfied and let $\la\!\notin\!{\cl}\{\Delta(\RR)\} \cup {\cl}\{d(\RR)\}$. If $0\notin\sess(S(\la))\cup\spt(S(\la))$, then $\la\in\Pi(T)$. 
\end{lemma}
\begin{proof}
Let $\la$ be as in the hypothesis. By Lemma~\ref{lem:kernel-smooth} and Theorem~\ref{thm:Schur-complement}, 
	\[
	\la\notin \sess(\cA)\cup\spt(\cA).
	\]
Denote by $\sapp(\cA)$ the approximate point spectrum of $\cA$. Because of the inclusion $\sapp(\cA)\subset\sess(\cA)\cup\spt(\cA)$, we thus obtain $\la\notin\sapp(\cA)$. Therefore,  
	\[
	\la\in\CC\setminus\sapp(\cA) \subset \CC\setminus\sapp(T)=\Pi(T),
	\]
where the inclusion is obvious since $T$ is a restriction of $\cA$.
\end{proof}

\begin{proposition}\label{prop:fin.dim.ext}
Let Assumptions~\ref{ass:A}, \ref{ass:C} be satisfied. Then $\cA$ is a finite-dimensional extension of $T$, that is,
	\begin{align}
	\dim \, \Dom(\cA) / \Dom(T)<\infty.
	\end{align}
\end{proposition}
\begin{proof}
For a given $\la\notin {\cl}\{\Delta(\RR)\} \cup {\cl}\{d(\RR)\}$, Lemma~\ref{lem:num.range} guarantees that there exists $\omega\in\CC$ such that 
\begin{align}\label{toshmat-S1}
\omega\notin\sess(S(\la))\cup\spt(S(\la)).
\end{align}
Since the closedness is preserved under bounded perturbations, the operators
	\begin{align*}
	\cA_{\omega} := \cA-\begin{pmatrix} 
	                    \omega  & 0 \\[1ex]
	                                 0 & 0 
	                    \end{pmatrix}, 
	                    \quad 
	                    \Dom(\cA_\omega) :=  \Dom(\cA),
	\end{align*}
and
	\begin{align*}
	S_\omega(\la) := S(\la)-\omega, \quad \Dom(S_\omega(\la)) := \Dom(S(\la))
	\end{align*}
are closed. In the sequel we will work with the following closed operators
	\begin{align*}
	&\cB_{I_j,\omega} := \cB_{I_j}-
							\begin{pmatrix} 
							\omega  & 0 \\ 
							0 & 0 
							\end{pmatrix},
							&&\Dom(\cB_{I_j,\omega}) :=  \Dom(\cB_{I_j}), &&\\
	&S_{I_j,\omega}(\la) := S_{I_j}(\la)-\omega,  &&\Dom(S_{I_j,\omega}(\la)) := \Dom(S_{I_j}(\la)), &&j=1,2,3.
	\end{align*}
By \eqref{toshmat-S1}, it is clear that $0\notin\sess(S_\omega(\la))\cup\spt(S_\omega(\la))$ and hence Lemma~\ref{lem:reg.field} implies that $\la\in\Pi(T_\omega)$, where 
\begin{equation}
T_\omega:=\cB_{I_1,\omega} \oplus \cB_{I_2,\omega} \oplus \cB_{I_3,\omega}, \quad \Dom(T_\omega):=\Dom(T). 
\end{equation}
In particular, this means that $\Pi(T_\omega)\neq\emptyset$. 
Since $S_\omega(\la)$ is Fredholm, it follows from Lemma~\ref{lem:kernel-smooth} and Corollary~\ref{cor:kernel-smooth}, that the quantities $\nl(\cA_\omega-\la)$ and $\df(\cA_\omega-\la)$ are both finite. 
Hence if we show that $\df(T_\omega-\la)<\infty$, then \cite[Theorem~III.3.1]{EE87} yields  
	\[
	\dim\Dom(\cA_\omega)/\Dom(T_\omega)=\nl(\cA_\omega-\la) + \df(T_\omega-\la) - \df(\cA_\omega-\la)<\infty,
	\]
and the claim of the proposition immediately follows as $\Dom(\cA)=\Dom(\cA_\omega)$ and $\Dom(T)=\Dom(T_\omega)$.

To this end, first note that $\df(T_\omega-\la) = \nl(T_\omega^*-\ov{\la})$ and 
	\begin{align}\label{Appendix-eshmat-1}
	\nl(T_\omega^*-\ov{\la}) = \nl(\cB_{I_1,\omega}^*-\ov{\la})+\nl(\cB_{I_2,\omega}^*-\ov{\la}) + \nl(\cB_{I_3,\omega}^*-\ov{\la}).
	\end{align}
It follows as in the proof of Proposition~\ref{prop:ALMS-O} that
	\begin{align}\label{Appendix-eshmat-2}
	\nl(\cB_{I_j,\omega}^*-\ov{\la}) &\leq \nl(S_{I_j,\omega}(\la)^*), \quad j=1,2,3,
	\end{align}
see \eqref{prop:ALMS-0-find.dim.last.step.1}. However, we know that $S_\omega(\la)$ is Fredholm and it is a finite dimensional extension\footnote{\,In fact, it is a $2m$-dimensional extension.} of the orthogonal sum $S_{I_1,\omega}(\la)\oplus S_{I_2,\omega}(\la)\oplus S_{I_3,\omega}(\la)$, see \cite[Section~IX.9]{EE87}. 
Therefore, the latter is also Fredholm and hence by the relations in \eqref{Appendix-eshmat-2} and \eqref{Appendix-eshmat-1}, we immediately get $\df(T_\omega-\la)<\infty$.
\end{proof}

\subsection{Proof of Theorem~\rm{\ref{thm:reg.part}}}
Let $z\!\in\!\Delta(\RR)$ be arbitrary. Then there exist $a,b\in\RR$ such that $z\in\Delta([a,b])$. Let $I_1=(-\infty,a]$, $I_2=[a,b]$, $I_3=[b,\infty)$ and let $T$, $\cB_{I_j}$ be the operators defined as in Subsection~\ref{sub.sec:mor.aux.results}. By Proposition~\ref{prop:fin.dim.ext}, the operator $\cA$ is a finite-dimensional extension of $T$. Since the essential spectrum is invariant with respect to finite-dimensional extensions, see  \cite[Corollary IX.4.2]{EE87}, we have $\sess(\cA) = \sess(T)$. Therefore, using \eqref{ess.spec.cB}, we obtain   
	\begin{align*}
	\Delta([a, b]) = \sess(\cB_{I_2}) \subset \sess(\cB_{I_1}) \cup \sess(\cB_{I_2}) \cup \sess(\cB_{I_3}) =\sess(T) \subset\sess(\cA).
	\end{align*}
Hence $\Delta(\RR)\subset\sess(\cA)$ and the claimed inclusion follows from the closedness of the essential spectrum.\qed

\section{Main result: analytic description of the essential spectrum}\label{sec:sharpened}

In this section we derive an explicit description of the essential spectrum of the closure $\cA$ of $\cA_0$ up to the set $\Lambda_\infty(d)$ defined to be the set of limit points at $\infty$ of the function $d:\RR\to\RR$ in \eqref{diff.expressions},
\begin{align}\label{Lambda-infty}
\Lambda_\infty(d) := \bigl\{\la\in\CC:\, \exists \, \{x_n\}_{n=1}^\infty\subset\RR \; {\rm{s.t.}} \; x_n\to\infty,\; d(x_n)\to\la, \; n\to\infty\bigr\}.
\end{align}

In our result a crucial role is played by the following assumption on the coefficients $p_j(\mydot,\la)$, $j\in\{0,1,2,\ldots,m\}$, of the Schur complement defined in \eqref{Schur.coeff}, which are formed out of the coefficients of the operator matrix $\cA_0$ in \eqref{A0}. 

\begin{ass}\label{ass:D} 
For every $\la \notin {\cl}\{\Delta(\RR)\}\cup 
	\Lambda_\infty(d)$ the following limits exist and are finite, 
	\begin{align}\label{coeff:aymp.const}
	p_j^{\pm}(\la) := \lim_{x\to\pm\infty}\frac{p_j(x,\la)}{p_m(x,\la)} , 
	\quad j=0,1,\ldots,m-1.
	\end{align}
\end{ass}

\begin{theorem}\label{thm:exp.description.singular}
	Let Assumptions~\ref{ass:A}, \ref{ass:B} and \ref{ass:D} be satisfied, let $\la \notin {\cl}\{\Delta(\RR)\}\cup \Lambda_\infty(d)$, and define the two polynomials 
	\begin{align}
	\label{polynomials:description}
	P_\la^{\pm}(\xi) := \xi^m+\sum_{j=0}^{m-1}p^{\pm}_j(\la)\xi^j, \quad \xi\in\RR.
	\end{align}
	Then $\la\in\sess(\cA)$ if and only if $P_\la^-(\xi)P_\la^+(\xi) = 0$ for some $\xi\in\RR$.
\end{theorem}

The proof of Theorem~\ref{thm:exp.description.singular} will be done in two steps: first for $\la \notin {\cl}\{\Delta(\RR)\}\cup {\cl}\{d(\RR)\}$ and then for those $\la$ such that $\la \in {\cl}\{d(\RR)\}\setminus\bigl({\cl}\{\Delta(\RR)\}\cup \Lambda_\infty(d)\bigr)$. In order to be able to use the result of the first step within the second, we need the following preparations.  

\begin{remark}\label{rem:D.infty}
If $\la\notin\Lambda_\infty(d)$, then there exists $x_\la>0$ such that, with $I_\la:=[-x_\la, x_\la]$, we have $\la\notin {\cl}\{d(\RR\setminus I_\la)\}$ and 
	\[
	\sup\limits_{x\in\RR\setminus I_\la}|(d(x)-\la)^{-1}|<\infty.
	\]
\end{remark}

Let $\la \notin \Lambda_\infty(d)$ be fixed and $x_\la$ be as in Remark~\ref{rem:D.infty}. Take $\varphi\in C_0^{\infty}(\RR)$ such that $\varphi(x) = 1$ for $|x-x_\la|\leq 1/4$ and $\varphi(x) = 0$ for $|x-x_\la|\geq 1/2$. Let $\psi\in C^{\infty}([x_\la,\infty), \CC)$ be given. Reflect the graph of $\psi$ with respect to the vertical axis $x=x_\la$ and let thus obtained graph correspond to a function $\psi_{\rm{ext}}$: 
\[
\psi_{\rm{ext}}(x) := \psi(x_\la+|x-x_\la|), \quad x\in\RR.
\]
For $\varepsilon>0$, set
\begin{align*} 
\psi_{\varepsilon}(x) := J_{\varepsilon}\ast(\varphi\psi_{\rm{ext}})(x) + ((1-\varphi)\psi_{\rm{ext}})(x), \quad x\in\RR,
\end{align*}
where 
\[
J_\varepsilon(x) := \frac{1}{\varepsilon}J\Bigl(\frac{x}{\varepsilon}\Bigr), \quad x\in\RR
\]
with a non-negative real-valued function $J\in C_0^{\infty}(\RR)$ such that 
\[
J(x) = 0 \quad \text{for} \quad |x|\geq 1, \quad \int_{\RR}J(x) \, \d x=1. 
\]
We call $\psi_{\rm{ext}}$ and $\psi_\varepsilon$, respectively, ``symmetric extension with respect to the vertical axis $x=x_\la$'' and ``extended $\varepsilon$-regularization'' of $\psi$. It is not difficult to see that $\psi_\varepsilon$ has the following properties
\begin{itemize}
	\item $\psi_{\varepsilon}\in C^{\infty}(\RR,\CC)$;
	\item $\psi_\varepsilon(x)=\psi_{\rm{ext}}(x)$, for all $x$ satisfying $|x-x_\la|\geq \varepsilon+1/2$;
	\item $\psi_\varepsilon$ uniformly converges to $\psi_{\rm{ext}}$ as $\varepsilon\searrow0$ on $\RR$.
\end{itemize}

\smallskip
\noindent
\textit{Proof of Theorem~{\rm{\ref{thm:exp.description.singular}}}}. \textbf{Step 1.} Let $\la \in \CC \setminus \big({\cl}\{\Delta(\RR)\}\cup {\cl}\{d(\RR)\} \big)$ be arbitrary. By Theorem~\ref{thm:Schur-complement}, we have
	\begin{align}
	\label{sess-equiv}
	\la \in \sess(\cA) \iff \ 0 \in \sess(S(\la)).
	\end{align}
Denote by $S^+(\la)$ and $S^{-}(\la)$ the restrictions of $S(\la)$ to the Sobolev spaces $H_0^{m}(\RR_+)$ and $H_0^{m}(\RR_-)$, respectively. Since $S(\la)$ is a finite dimensional extension of the orthogonal sum $S^{+}(\la)\oplus S^{-}(\la)$, we have  
	\begin{align}\label{sess-equiv-1}
	\la\in\sess(\cA) \quad \Longleftrightarrow \quad 0 \in \sess(S(\la))=\sess(S^{+}(\la))\cup\sess(S^{-}(\la)).
	\end{align}
It follows by Assumptions~\ref{ass:A} and \ref{ass:D} that the differential operator 
	\begin{equation}\label{S_1}
	S^+_1(\la) := \frac{1}{p_m(\mydot,\la)} S^{+}(\la) = P_\la^{+}\Bigl(-\ii\dx\Bigr), \quad \Dom(S_1(\la))=H_0^m(\RR_+)
	\end{equation}
satisfies all the conditions of \cite[Corollary~IX.9.4]{EE87}. Therefore \cite[(9.19)]{EE87} for $\sigma_{{\rm{e}}k}$ with $k=3$ applies and yields
	\begin{equation}\label{p-h-1}
	\sess\bigl(S^+_1(\la)\bigr) = \left\{P_\la^{+}(\xi) : \xi\in\RR\right\}.
	\end{equation}
Since $|p_m(\mydot,\la)|$ is uniformly positive on $\RR$, \eqref{S_1} and \eqref{p-h-1} imply that
	\begin{align}\label{ess.spec.S+}
	\sess(S^{+}(\la)) = \sess\bigl(S^+_1(\la) \bigr) =\left\{P_\la^{+}(\xi) : \xi\in\RR\right\}.
	\end{align}
In the same way we obtain 
		\begin{align}\label{ess.spec.S-}
		\sess(S^{-}(\la)) = \sess (S^-_1(\la))=\left\{P_\la^{-}(\xi) : \xi\in\RR\right\},
		\end{align}	
if we use the unitary transformation
	\[
	U:L^2(\RR_-) \to L^2(\RR_+), \quad (U\phi)(x):=\phi(-x).
	\]
Altogether, from \eqref{ess.spec.S+}, \eqref{ess.spec.S-} and \eqref{sess-equiv-1} the claim follows.

\smallskip
\noindent
\textbf{Step 2.}
Let $\la \in {\cl}\{d(\RR)\}\setminus\bigl({\cl}\{\Delta(\RR)\}\cup \Lambda_\infty(d)\bigr)$ be arbitrary and $x_\la>0$ be chosen as in Remark~\ref{rem:D.infty}. Let $I_1=(-\infty,-x_\la-1]$, $I_2=[-x_\la-1,x_\la+1]$, $I_3=[x_\la+1,\infty)$ and denote by $\cA_{I_j}$ the closures of the restrictions of $\cA_0$ to $C_0^\infty(I_j) \oplus C_0^\infty(I_j)$, $j=1,2,3$. Then it follows by  Proposition~\ref{prop:fin.dim.ext} that $\cA$ is a finite-dimensional extension of the orthogonal sum $\cA_{I_1}\oplus\cA_{I_2}\oplus\cA_{I_3}$ and hence by \cite[Corollary IX.4.2]{EE87}, we have
\begin{align}\label{improvement:decomposition}
\sess(\cA)=\sess(\cA_{I_1})\cup\sess(\cA_{I_2})\cup\sess(\cA_{I_3}).
\end{align}
By Proposition~\ref{prop:ALMS-O}, we have   
\begin{align}\label{improvement:Langer-1}
\sess(\cA_{I_2})= \Delta(I_2).
\end{align}
Since $\la \notin {\cl}\{\Delta(\RR)\} \supset \Delta(I_2)$, we therefore have by \eqref{improvement:decomposition} and \eqref{improvement:Langer-1},
\[
\la\in\sess(\cA) \quad \Longleftrightarrow \quad \la\in\sess(\cA_{I_1})\cup\sess(\cA_{I_3}).
\]
In the sequel, we show  that 
\begin{align}\label{relation:I_3}
\la\in \sess(\cA_{I_3}) \quad \Longleftrightarrow  \quad P^+_\la(\xi)=0 \quad \text{for some} \quad \xi\in\RR.
\end{align}

Note that the restrictions to $I_3$ of the function $\Delta$ in \eqref{defDelta} as well as of the coefficient functions of $\cA_0$ in \eqref{A0} are all $C^\infty$-functions. Below we will work with their symmetric extensions with 
respect to the vertical axis $x=x_\la$ and also with the extended $\varepsilon$-regularizations. 

Observe that, for $\varepsilon\in(0, 1/2)$,
\[
\Delta_\varepsilon([x_\la+1,\infty))=\Delta_{\rm{ext}}([x_\la+1,\infty))=\Delta([x_\la+1,\infty)).
\]
Next, we prove that, for some $\varepsilon^*>0$,  
\begin{align}\label{lem:reg.epsilon}
\la \notin {\cl}\{\Delta_\varepsilon(\RR)\} \cup {\cl}\{d_\varepsilon(\RR)\}, \quad \varepsilon\in(0,\varepsilon^*).
\end{align}
To this end, recall that $\la\!\notin\!{\cl}\{d((-\infty,-x_\la)\cup(x_\la,\infty))\}$ by Remark~\ref{rem:D.infty}. Therefore, $\la\notin{\cl}\{d_{\rm{ext}}(\RR)\}$ and hence there exists 
$\delta>0$ such that $\dist(\la,d_{\rm{ext}}(\RR)) \geq \delta$. 

On the other hand, since $d_\varepsilon$ converges to $d_{\rm{ext}}$ uniformly as $\varepsilon\searrow 0$ on $\RR$, there exists $\varepsilon_1>0$ such that $|d_\varepsilon(x)-d_{\rm{ext}}(x)|<\delta/2$ for all $x\in\RR$ whenever $\varepsilon\in(0,\varepsilon_1)$. Therefore, by the triangle inequality, we have
\begin{align*}
|d_\varepsilon(x)-\la| \geq |d_{\rm{ext}}(x)-\la|-|d_\varepsilon(x)-d_{\rm{ext}}(x)|\geq \delta/2, \quad x\in\RR, \quad \varepsilon\in(0,\varepsilon_1).
\end{align*}
Hence $\dist\bigl(\la, {\cl}\{d_\varepsilon(\RR)\}\bigr) \geq \delta/2$, that is, $\la\notin{\cl}\{d_\varepsilon(\RR)\}$ whenever $\varepsilon\in(0,\varepsilon_1)$. The proof of $\la\notin {\cl}\{\Delta_\varepsilon(\RR)\}$, for all $\varepsilon\in(0,\varepsilon_2)$ with some $\varepsilon_2>0$, is the 
same and follows from $\la \notin {\cl}\{\Delta(\RR)\}$ using the fact that $\Delta_\varepsilon(x)$ converges to $\Delta_{\rm{ext}}(x)$ uniformly as $\varepsilon\searrow 0$ on $\RR$. Consequently, \eqref{lem:reg.epsilon} holds with $\varepsilon^*=\min\{\varepsilon_1, \varepsilon_2\}$.

Now take any $\varepsilon>0$ such that $\varepsilon<\min\{\varepsilon^*, 1/2\}$ and consider the operator matrix 
\begin{equation}\label{A01}
\begin{aligned} 
&\cB_{0,\varepsilon}  := 
\left(\begin{array}{cc} 
\ds\sum_{\alpha=0}^{m} a_{\alpha,\varepsilon}D^{\alpha}\: & \ds\sum_{\beta=0}^{n} b_{\beta,\varepsilon}D^{\beta}\\[3.5ex]
\ds\sum_{\gamma=0}^{k} c_{\gamma,\varepsilon}D^{\gamma} \: & D_\varepsilon  
\end{array}\right), \\[1.5ex]
&\Dom(\cB_{0,\varepsilon})  :=  C_0^\infty(\RR) \oplus C_0^\infty(\RR),
\end{aligned}
\end{equation}
where $a_{\alpha,\varepsilon}$, $b_{\beta,\varepsilon}$, $c_{\gamma,\varepsilon}$ and $d_\varepsilon$ stand for the extended $\varepsilon$-regularizations of the coefficients $a_\alpha$, $b_\beta$, $c_\gamma$ and $d$, respectively of $\cA_0$. We denote the closure of $\cB_{0,\varepsilon}$ by $\cB_{\varepsilon}$. Moreover, we denote by $\cB_{\varepsilon,I_j}$ the closures of the restrictions of $\cB_{\varepsilon,0}$ to $C_0^\infty(I_j) \oplus C_0^\infty(I_j)$, $j=1,2,3$. 

It follows by Proposition~\ref{prop:fin.dim.ext} that $\cB_{\varepsilon}$ is a finite-dimensional extension of the orthogonal sum $\cB_{\varepsilon,I_1}\oplus\cB_{\varepsilon,I_2}\oplus\cB_{\varepsilon,I_3}$ and hence 
by \cite[Corollary IX.4.2]{EE87},
\[
\sess(\cB_{\varepsilon})=\sess(\cB_{\varepsilon, I_1})\cup\sess(\cB_{\varepsilon, I_2})\cup\sess(\cB_{\varepsilon, I_3}).
\]
By Proposition~\ref{prop:ALMS-O}, we have $\sess(\cB_{\varepsilon, I_2})=\Delta_\varepsilon(I_2)$. On the other hand, \eqref{lem:reg.epsilon} yields $\la\notin{\cl}\{\Delta_\varepsilon(\RR)\} \supset {\cl}\{\Delta_\varepsilon(I_2)\}$ and thus
\begin{align}\label{eshmat-gluing}
\la\in\sess(\cB_\varepsilon) \quad \Longleftrightarrow \quad \la\in\sess(\cB_{\varepsilon, I_1})\cup\sess(\cB_{\varepsilon, I_3}). 
\end{align}
Observe that by our construction $\cA_{I_3}\equiv\cB_{\varepsilon,I_3}$ and that $\cB_{\varepsilon,I_1}$ and $\cB_{\varepsilon,I_3}$ are unitarily equivalent by the unitary transformation 
\[
U:L^2(\RR)\to L^2(\RR), \quad (U\phi)(x) := \phi(-x).
\]
Therefore, $\sess(\cA_{I_3})=\sess(\cB_{\varepsilon, I_1})=\sess(\cB_{\varepsilon, I_3})$ and hence \eqref{eshmat-gluing} implies that 
\begin{align}\label{A-B_e}
\la\in\sess(\cB_\varepsilon) \quad \Longleftrightarrow \quad \la\in \sess(\cA_{I_3}).
\end{align}
Due to \eqref{lem:reg.epsilon} it is easy to see that the operator $\cB_{\varepsilon}$ satisfies all the hypotheses of the first step. In particular, the limits 
\begin{align}\label{coeff:aymp.const.varepsilon}
	p_{j,\varepsilon}^{\pm}(\la) := \lim_{x\to\pm\infty}\frac{p_{j,\varepsilon}(x,\la)}{p_{m,\varepsilon}(x,\la)} , 
	\quad j=0,1,\ldots,m-1,
\end{align}
corresponding to \eqref{coeff:aymp.const} for $\cB_\varepsilon$ exist, where $p_{j,\varepsilon}(\mydot,\la)$, $j\in\{0,1,\ldots,m\}$, stand for the coefficients of the first Schur complement of $\cB_\varepsilon$. By the construction of the extended $\varepsilon$-regularization, we clearly have
\begin{align}
p_{j,\varepsilon}^{\pm}(\la)=p_j^{+}(\la), \quad j=0,1,\ldots,m-1.
\end{align}
Therefore, we conclude from Step 1 that
\begin{align}\label{B_e}
\la\in\sess(\cB_{\varepsilon}) \quad \Longleftrightarrow \quad P^+_\la(\xi)=0 \quad \text{for some} \quad \xi\in\RR.
\end{align}
Now the claim in \eqref{relation:I_3} follows from \eqref{B_e} and \eqref{A-B_e}. 

On the other hand, applying the same arguments as above, we obtain 
\[
\la\in \sess(\cA_{I_1}) \quad \Longleftrightarrow  \quad P^-_\la(\xi)=0 \quad \text{for some} \quad \xi\in\RR,
\]
completing the proof of the theorem.\qed

\medskip

From Theorems~\ref{thm:reg.part} and \ref{thm:exp.description.singular} we come to the following conclusion, which is the main result of the paper.

\begin{theorem}\label{thm:main.result}
	Let Assumptions~\ref{ass:A}, \ref{ass:C} and \ref{ass:D} be satisfied. Then for the closure $\cA$ of the operator $\cA_0$ in \eqref{A0}, we have
	\begin{align}
	\sess(\cA)  = \sessreg(\cA) \cup \sesssing(\cA),
	\end{align}
	where $\sessreg(\cA):={\cl}\{\Delta(\RR)\}$, and
	\begin{align*}
	\sesssing(\cA)\setminus \Lambda_\infty(d) := \bigl\{\la\!\in\!\CC\!\setminus\!\bigl({\cl}\{\Delta(\RR)\} \cup \Lambda_\infty(d)\bigr): \exists\, \xi\in\RR \;\, {\rm{s.t.}} \;\, P^+_\la(\xi)P^-_\la(\xi) = 0\bigr\}.
	\end{align*}
\end{theorem}

\begin{example}\label{example}

In the Hilbert space $L^2(\RR) \oplus L^2(\RR)$, we consider the matrix differential operator

\begin{equation*}\label{ex:A0}
\begin{aligned} 
& \cA_0  := \left(\begin{array}{cc} 
D^4+\ii D^2+\frac{x^2}{x^2+1}& \:D+\frac{\cos(x)}{\sqrt{1+x^2}}\\[5ex]
\ii D^3+\frac{x^2}{\ii+x^2} & \:{\rm{e}}^{-x^2/2}+\frac{\ii}{1+x^2} 
\end{array}\right), \quad \Dom(\cA_0) :=  C_0^\infty(\RR) \oplus C_0^\infty(\RR),
\end{aligned}
\end{equation*}
where as in \eqref{diff.expressions}, $D$ stands for $-\ii\,{\rm{d}}/{\rm{d}}x$. Clearly, $\Lambda_\infty(d)=\{0\}$, see \eqref{Lambda-infty}, and it is easy to check that Assumptions~\ref{ass:A}, \ref{ass:C} and \ref{ass:D} are satisfied and Theorem~\ref{thm:main.result} can be applied. The function $\Delta:\RR\to\CC$ defined in \eqref{defDelta} is given by
\[
\Delta(x)= {\rm{e}}^{-x^2/2}-\frac{\ii x^2}{1+x^2}, \quad x\in\RR.
\]
Observe that $-\ii\in{\cl}\{\Delta(\RR)\}$ since $\lim_{|x|\to\infty}\Delta(x)=-\ii$. Moreover, the polynomials in \eqref{polynomials:description} are given by
\[
P^{\pm}_\la(\xi) = \xi^4+\frac{\ii\la}{\ii+\la}\xi^2+\frac{1}{\ii+\la}\xi+\frac{\la-\la^2}{\ii+\la}, \quad \la\in\CC\setminus\{-\ii\}.
\]
Theorem~\ref{thm:main.result} yields that, for the closure $\cA$ of $\cA_0$,
\begin{equation}\label{sess-in-example}
\begin{aligned}
\sess(\cA) = & \bigl\{{\rm{e}}^{-x^2/2}-\ii x^2/(1+x^2) : \: x\in\RR\bigr\} \\
& \cup \bigl\{\la\!\in\!\CC\!\setminus\!\{-\ii\} \,:\: \exists\,\xi\in\RR, (\ii+\la)\xi^4+\ii\la\xi^2+\xi+\la-\la^2=0 \bigr\}. 
\end{aligned}
\end{equation}
The essential spectrum is shown in the Figure~\ref{fig:total}. Here the blue curves correspond to the branch of the essential spectrum due to the singularity at infinity (second set in \eqref{sess-in-example}) while the red one corresponds to the essential spectrum due to the violation of the ellipticity in the sense of Douglis and Nirenberg (first set in \eqref{sess-in-example}). Observe that the exceptional set $\Lambda_\infty(d)=\{0\}$ is contained in the singular part of the essential spectrum. Moreover, one can show that the second curve in blue in the right upper quadrant is unbounded and extends to infinity. 

\begin{figure}[h!]
	\centering
	\fbox{\includegraphics[width=0.45\textwidth]{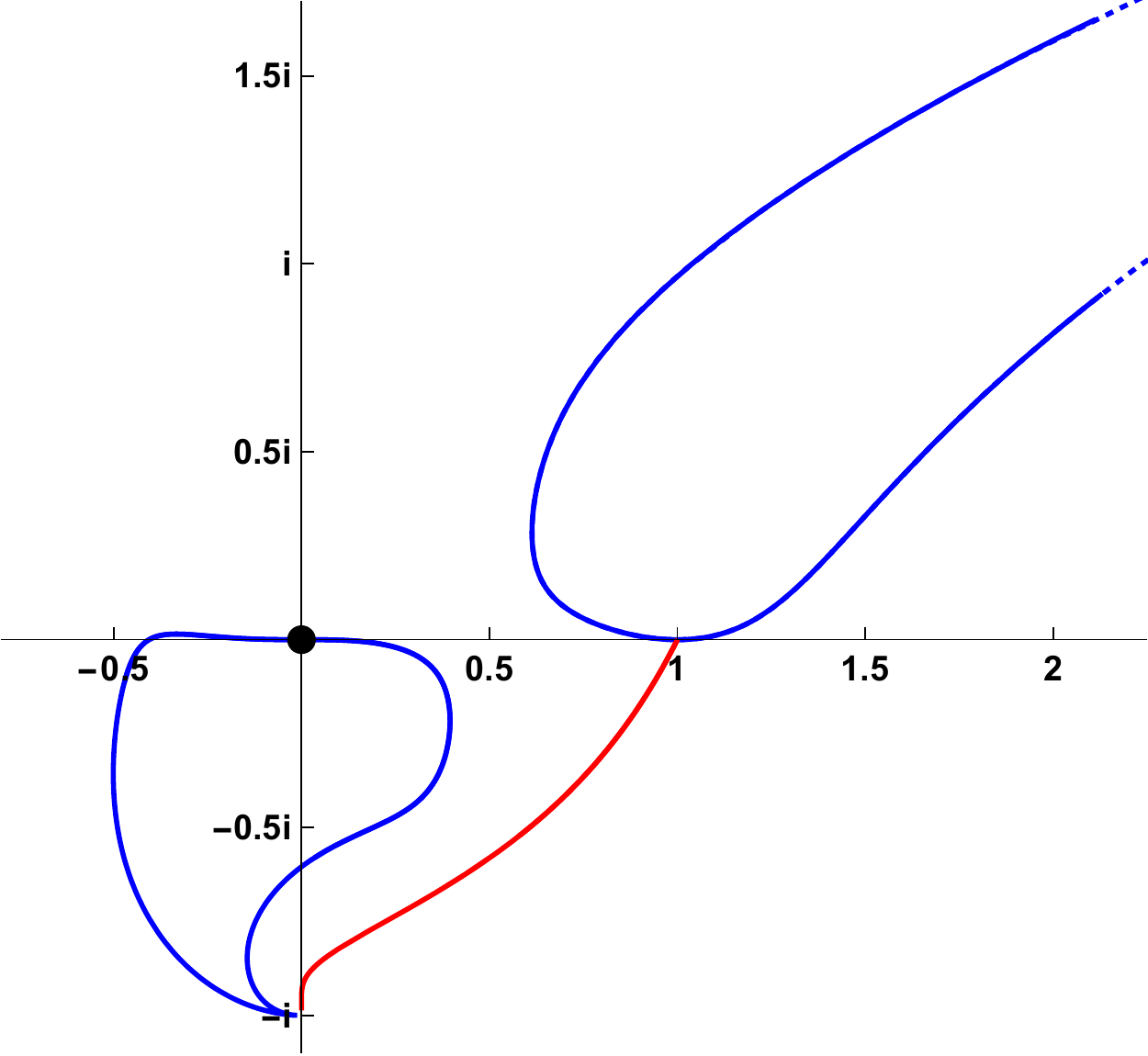}}
	\caption{\small{Essential spectrum of the closure of the operator $\cA_0$ in Example~\ref{example}: blue part caused by singularity at infinity and red part due to violation of Douglis-Nirenberg ellipticity.}}\label{fig:total}
\end{figure}
\end{example}

The following example demonstrates that the two parts of the essential spectrum can be disjoint\footnote{\,The question whether the two parts of the essential spectrum are always adjoined to each other was asked by Prof. Pavel Kurasov at the conference ``Spectral Theory and Applications'', Stockholm, March, 2016. We note that the answer was always affirmative in the models of the previous studies.}.
\begin{example}
In the Hilbert space $L^2(\RR) \oplus L^2(\RR)$, we consider the matrix differential operator
\begin{equation}
\begin{aligned} 
\ds\cA_{0} & := \left(\begin{array}{cc} 
\ds -\frac{\d^2}{\d x^2}& \ds-\dx\\[3ex]
\ds\dx & \ds\:-x^2 
\end{array}\right),
\end{aligned} \quad \Dom(\cA_{0}) :=  C_0^\infty(\RR) \oplus C_0^\infty(\RR).
\end{equation}
Clearly, $\Lambda_\infty(d)=\emptyset$ and it is easy to check that Assumptions~\ref{ass:A}, \ref{ass:C}, \ref{ass:D} are satisfied. The function $\Delta$ in \eqref{defDelta} is given by $\Delta(x)=-x^2-1$, $x\in\RR$, and hence the Douglis-Nirenberg ellipticity of $\cA_0-\la$ is violated if and only if $\la\in(-\infty, -1]$. The coefficients of the Schur complement 
in \eqref{Schur.coeff} are given by 
\begin{align*}
p_0(x,\la) = -\la, \quad p_1(x,\la) = -\frac{2\ii x}{(x^2+\la)^2}, \quad p_2(x,\la) = 1+\frac{1}{x^2+\la}
\end{align*} 
for $x\in\RR$, $\la\in\CC\setminus\RR_-$. It is easy to check that all assumptions of Theorem~\ref{thm:main.result} are satisfied, in particular, 
\[
\frac{p_0(x,\la)}{p_2(x,\la)}=-\la+{\rm{o}}(1), \quad \frac{p_1(x,\la)}{p_2(x,\la)}={\rm{o}}(1), \quad |x|\to \infty,
\]
and thus the polynomials $P_\la^\pm(\xi)$ in \eqref{polynomials:description} are given by $
P_\la^\pm(\xi)=\xi^2-\la$, $\xi\in\RR$. Therefore, Theorem~\ref{thm:main.result} yields
 \[
\sess(\cA)=(-\infty, -1] \cup [0, \infty)
\]
for the closure $\cA$ of $\cA_0$. Obviously, the two parts of the essential spectrum are not adjoined to each other.
\end{example}


\medskip
\noindent
{\bf Acknowledgements.} \ I gratefully acknowledge the support of the \emph{Swiss National Science Foundation}, SNF, grant no.\ $200020\_146477$. This work was done during my PhD studies at the Institute of Mathematics in Bern, Switzerland. I am deeply grateful to Prof. Dr. Christiane Tretter and Dr. Petr Siegl for many stimulating discussions and very helpful suggestions.

\hfill

{\small
	\bibliographystyle{acm}
	\bibliography{ref_NSA}
}

\end{document}